\documentclass[11pt]{article}
\usepackage{fullpage}

\usepackage{graphicx}
\usepackage{amssymb}

\usepackage{lineno}

\usepackage{enumerate}

\usepackage{amsmath,amsfonts,amsthm}
\usepackage{amsmath,amsfonts}
\usepackage{url}
\newtheorem{theorem}{Theorem}

\newtheorem{lemma}{Lemma}
\newtheorem{definition}{Definition}
\newtheorem{cor}{Corollary}
\newtheorem{prop}{Proposition}

\newtheorem{example}{Example}

\newcommand{\Tn}{\mathbb{T}_n}    
\newcommand{\Or}{\mathcal{O}}          
\newcommand{\R}{\mathbb{R}}    
\newcommand{\inputT}{\mathcal{T}}   
\newcommand{\Ell}{\mathcal{L}}  
\newcommand{\ssd}{\delta}  

\begin{document}

\title{Properties for the Fr\'echet Mean in Billera-Holmes-Vogtmann Treespace}


\author{Maria Anaya\thanks{Stoney Brook University} 
\and Olga Anipchenko-Ulaj\thanks{City College of New York, CUNY }
\and Aisha Ashfaq\footnotemark[2]
\and Joyce Chiu\thanks{Brooklyn College, CUNY}
\and Mahedi Kaiser\thanks{General Motors}
\and Max Shoji Ohsawa\footnotemark[3]
\and Megan Owen\thanks{Department of Mathematics, Lehman College, CUNY}
\and Ella Pavlechko\thanks{North Carolina State University}
\and Katherine St.~John\thanks{Department of Computer Science, Hunter College, CUNY and Division of Invertebrate Zoology, American Museum of Natural History, New York, NY 10024}
\and Shivam Suleria\footnotemark[3]
\and Keith Thompson\thanks{College of Staten Island, CUNY}
\and Corrine Yap\thanks{Rutgers University}
}

\maketitle



\begin{abstract}
The Billera-Holmes-Vogtmann (BHV) space of weighted trees can be embedded in Euclidean space, but the extrinsic Euclidean mean often lies outside of treespace.  Sturm showed that the intrinsic Fr\'echet mean exists and is unique in treespace.  This Fr\'echet mean can be approximated with an iterative algorithm, but bounds on the convergence of the algorithm are not known, and there is no other known polynomial algorithm for computing the Fr\'echet mean nor even the edges present in the mean.  We give the first necessary and sufficient conditions for an edge to be in the Fr\'echet mean.  The conditions are in the form of inequalities on the weights of the edges.  These conditions provide a pre-processing step for finding the treespace orthant containing the Fr\'echet mean.  This work generalizes to orthant spaces.
\end{abstract} 






\section{Introduction}

Evolutionary histories for a set of species are often represented by tree structures.  The leaves of the tree represent the living species, and the internal nodes represent the hypothetical ancestors.  The addition of weights to the edges represent the amount of time or evolutionary change that has occurred between nodes or confidence in the edge.  While edge weights makes the model more complex, it simplifies the comparison of trees \cite{kendall2015,kuhner1994}. Many of the popular metrics for comparing unweighted trees are based on tree rearrangement operations and are computationally hard to compute \cite{allenSteel,bordewichSemple,nni,hickey2008}.  Billera, Holmes, and Vogtmann \cite{bhv} introduced a space for weighted trees that views trees as vectors of their branch weights, called the {\em BHV treespace}.  This space is non-Euclidean but has unique geodesics (shortest paths between points), because it is globally non-positively curved (CAT(0)).  Owen and Provan \cite{Owen2011} gave a polynomial time algorithm to compute geodesics and distances in this space.  In addition to being a natural space for comparing phylogenetic, or evolutionary, trees, it has showed promise for classifying features of branching patterns in the airways of the lungs and arteries in the brain \cite{airwayIPMI,airwayIEEE,skwererBrain}.

The continuous treespace provides a promising setting for statistics on sets of trees.  Work in this direction includes:  principal components analysis (PCA) \cite{airwayIPMI,nyePCA1,nyePCA2,nye2017locus}, random walks \cite{nyeRandomWalksBHV}, and other  measures of uncertainty \cite{willis2016,willisBell2016}.  Many of these approaches require computing a mean or ``average" of a set of trees.  In Euclidean space, there are multiple ways to compute the mean of a set of points that all yield equivalent results.  In BHV treespace, taking the coordinate-wise average as for the Euclidean mean, can yield a new vector that does not correspond to a tree.  Thus, in BHV treespace, the Fr\'echet mean, which minimizes the sum of squared distances to the input trees within treespace, is used.  The Fr\'echet mean is unique on globally non-positively curved spaces \cite{sturm2003}, such as the BHV space, and there are iterative algorithms that converge to the mean \cite{bacakMeans,miller2015,  skwererMeans}.   However, there are no known bounds on the convergence rate.  Like other measures of central tendency for trees, the Fr\'echet mean exhibits non-Euclidean behaviors (such as ``stickiness'' \cite{hotzStickyOpenBooks}), but it is more likely to yield binary (fully resolved) trees on biological datasets than the well-known majority-rules consensus tree \cite{BrownOwen}.  Whether the Fr\'echet mean in BHV treespace can be computed in polynomial time is an open question, and this paper works towards answering this question in the affirmative.



There is a geometric characterization of the Fr\'echet mean in BHV treespace \cite{CLTcomplete}, but there is no combinatorial characterization of the mean, which seems to be necessary for an exact polynomial time algorithm.  
We give the first necessary and sufficient conditions for an edge to be in the Fr\'echet mean.  
We derive inequalities on the edges weights of the input trees from properties of the ``log map'' \cite{CLTcomplete,bardenT4,bardenTn} and the characterization of geodesics in treespace.  The log map gives a projection of the BHV space that can be used to ``unfold'' geodesics into Euclidean space (described in Section~\ref{sec:logmap}).
These conditions provide a pre-processing step for finding the treespace orthant containing the Fr\'echet mean.  This work generalizes to orthant spaces.





\section{Preliminaries}

In this section, we briefly describe trees used for evolutionary histories, the Billera-Holmes-Vogtmann (BHV) space of continuous trees, and a helpful technique for unfolding geodesics in the BHV treespace into Euclidean space.  More details can be found in \cite{shapeOfPhylo} or \cite{sempleSteelBook}.

\subsection{Trees}

Let $\Ell$ be a set of labels, such as the names of species. 
A phylogenetic tree $T$ is a directed acyclic graph in which all internal nodes have degree 3 or higher, and the leaves are in bijection with the labels in $\Ell$.  A phylogenetic tree is called \emph{binary} when all internal nodes have exactly degree 3, and \emph{non-binary}, \emph{degenerate}, or \emph{unresolved} otherwise.  For this paper, we consider the trees to be unrooted, but the results hold for rooted trees, in which one of the leaves $\Ell$ is distinguished as the root.  Each edge of tree $T$ is assigned a weight (or length), which is a positive real number and could correspond to the mutation rate along that edge or confidence in the existence of the edge.  Let $|e|_T$ be the weight of edge $e$ in tree $T$.

An edge that has a leaf as an endpoint is a \emph{pendant edge}.  An edge that is not a pendant edge is called an \emph{interior edge}.  We are primarily concerned with interior edges, as the pendant edges are shared by all trees, leading to a straightforward way to account for them in the mean tree.  (See Proposition~\ref{prop:edges_common_to_all_trees} in Section~\ref{s:mean_intro}).  Unless noted, a edge will mean an interior edge. 

A \emph{split}, $A|B$, is a partition of the leaf set $\Ell$ into two parts (a `bipartition'), where $A \subseteq \Ell$ and $B = \Ell \setminus A$.  Each edge of the tree $T$ divides the leaves into two parts, namely the leaves in the subtree on one side of the edge and the leaves in the subtree on the other side of the edge.  While strictly speaking the corresponding edge in a tree, and not the split itself, has a weight, we will abuse notation and use $|s|_T$ to represent the weight of the edge corresponding to split $s$ in tree $T$, with this value being 0 if split $s$ is not in $T$.  As with the edges, we are primarily interested in splits corresponding to interior edges, which are all splits with at least 2 elements in each part of the bipartition.  Unless noted, a split will mean an interior split.  Let $\Sigma$ denote the set of all possible (interior) splits on $\Ell$.

Two splits $s_1 = Y_1 | Y_2$ and $s_2 = Z_1 | Z_1$ are \emph{compatible} if at least one of $Y_1 \cap Z_1$, $Y_1 \cap Z_2$, $Y_2 \cap Z_1$, and $Y_2 \cap Z_2$ is empty.  Intuitively, two different splits are compatible if they can exist in the same tree.  Two different splits that are not compatible are \emph{incompatible}.  A split $s$ is trivially compatible with itself.  Unless noted, a compatible split refers to splits that are non-trivially compatible.

Let $E(T)$ to be the set of weighted edges (or splits, if clear from the context) in $T$. If $\inputT$ is a set of trees, let $E(\inputT)$ be the set of unique splits in the trees of $\inputT$.  Let $E \in \Sigma$ be a set of mutually compatible splits.  Then define $C(E)$ to be the set of splits that are compatible with all splits in $E$, and define $X(E)$ be the set of splits that are incompatible with at least one split in $E$.  That is, 
$$
    C(E) = \{ s \in \Sigma : \forall e \in E, s \text{ is compatible with } e\},
$$
and 
$$
    X(E) = \{ s \in \Sigma: \exists e \in E \text{ such that } s \text{ and } e \text{ are incompatible} \}.
$$
To streamline notation, we will use $C(T)$ and $X(T)$ to represent $C(E(T))$ and $X(E(T))$, and for any edge or split $e$, we will use $C(e)$ and $X(e)$ to represent $C(\{e\})$ and $X(\{e\})$.  

\begin{figure}
\begin{center}
\begin{tabular}{ccc}
\includegraphics[height=2in]{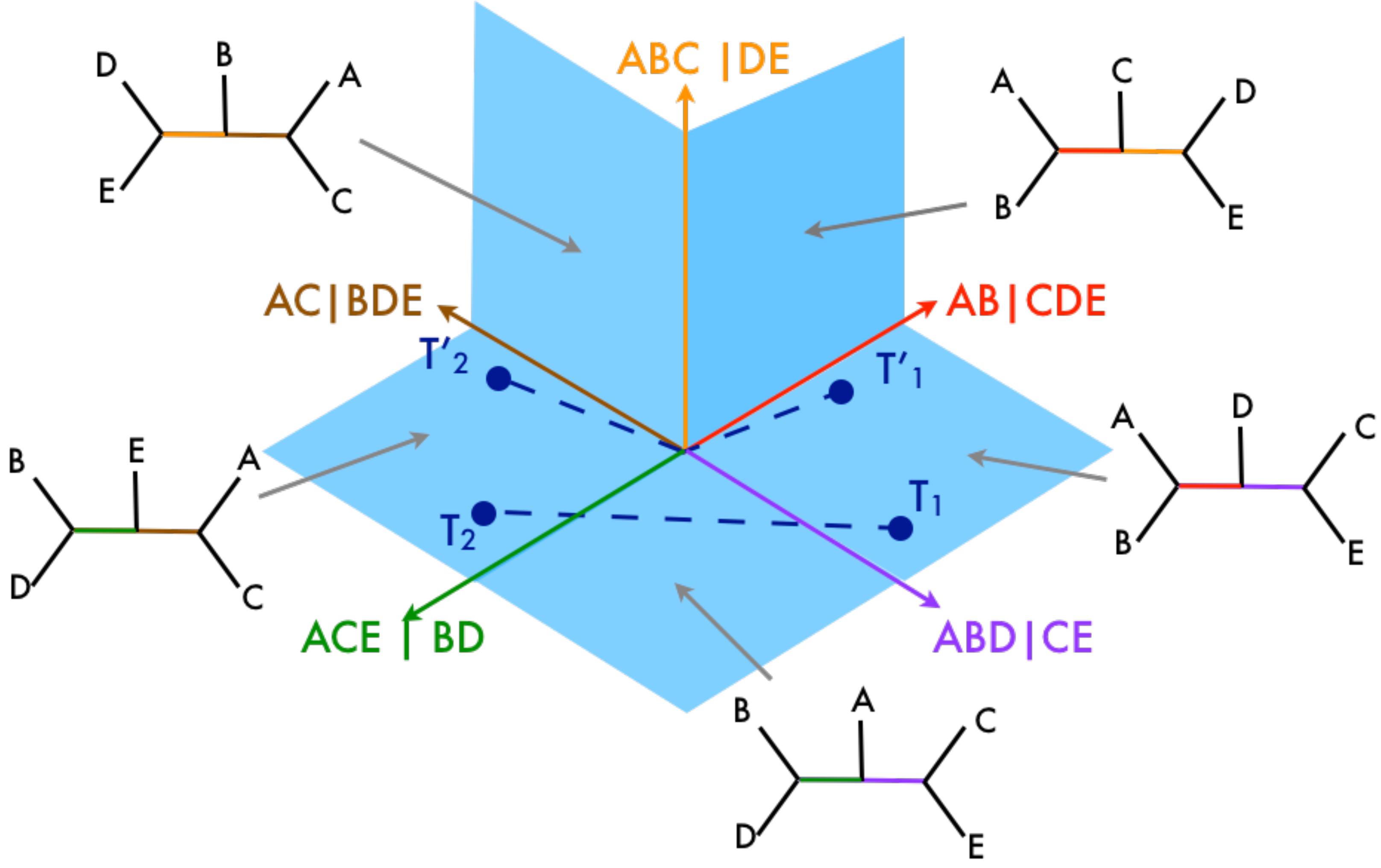}
&&  
\includegraphics[height=2in]{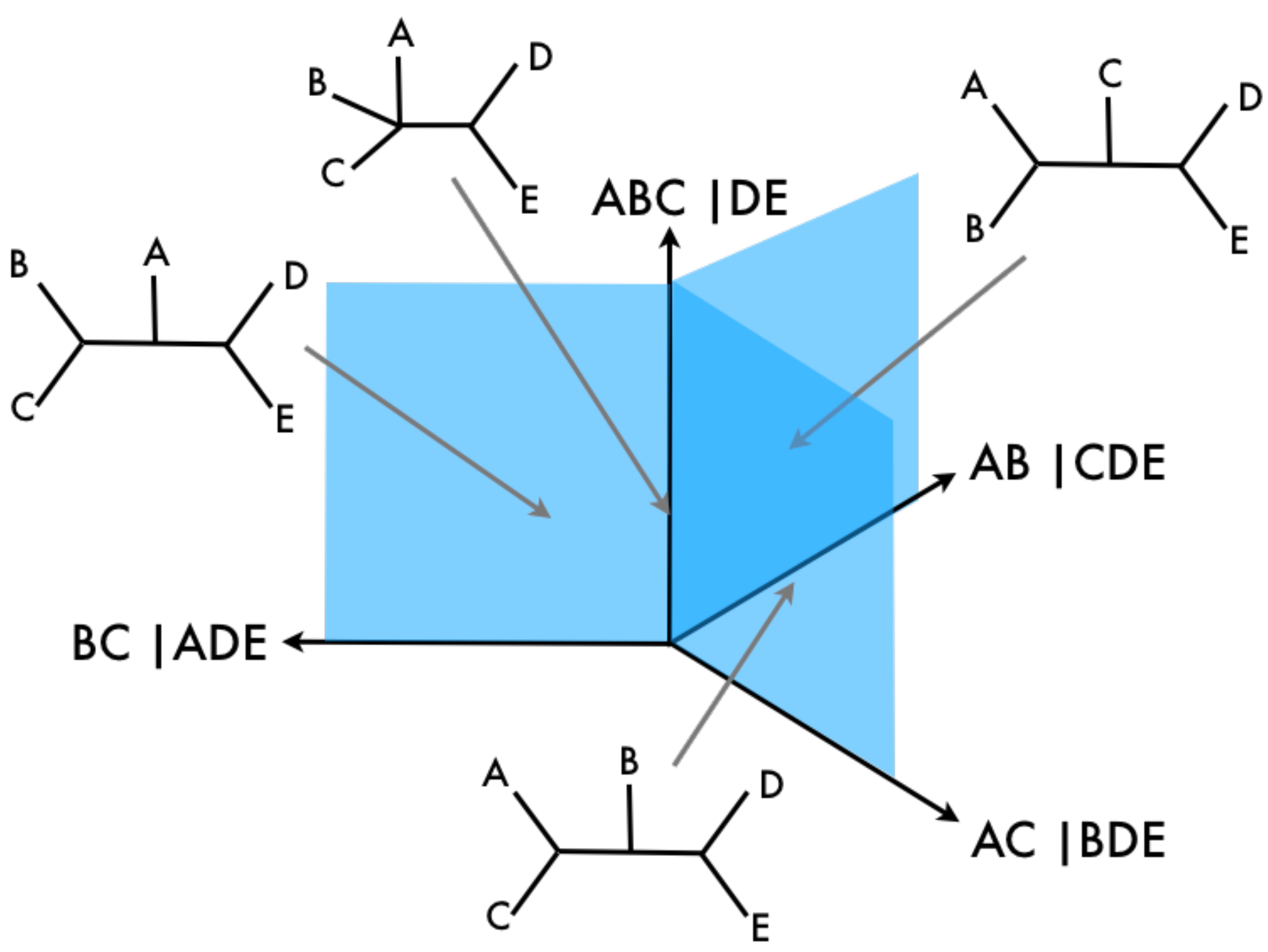}\\
\\
    (a) && (b)
\end{tabular}

\end{center}

    \caption{ a) Five of the 15 7-dimensional orthants in $\mathcal{T}_5$, with only the two dimensions corresponding to the two interior edges shown for each orthant.  This 5-dimensional figure has further been embedded into 3 dimensions for ease of visualization.  The geodesics between the pair $T_1$ and $T_1'$, and the pair $T_2$ and $T_2'$ are shown as dashed lines.  Trees $T_1$ and $T_2$, and trees $T_1'$ and $T_2'$ are each in the same orthant, but the geodesics between them differ by which orthant interiors they cross due to differences in edges lengths of the endpoint trees.
    b) Three orthants that share the common split $ABC|DE$. Again for each orthant, only the two dimensions corresponding to the interior edges are shown.  The tree on the axis corresponding to split $ABC|DE$ only contains that interior edge, and thus is degenerate.  }
\label{fig:bhv}
\end{figure}

\subsection{BHV Treespace}

An elegant way to organize the phylogenetic trees on $n$ leaves is via the Billera-Holmes-Vogtmann (BHV) treespace, $\Tn$ \cite{bhv}.  The BHV treespace is composed of \emph{orthants}, which are the non-negative part of Euclidean space and generalize quadrants and octants.  Each binary tree corresponds to the interior of an orthant that is a copy of $\mathbb{R}^{n-3}_+$, 
where the coordinates represent the weight of each edge (see Figure~\ref{fig:bhv}). Trees that are not binary will have fewer than $n-3$ positively weighted edges, and thus will lie on the boundaries of these top dimensional orthants.  Let $\Or(T)$ be the minimal, or smallest dimensional, orthant containing the tree $T$ in its interior. Note that lower dimemsional orthants lie on the boundary of the top dimensional orthants (see Figure~\ref{fig:bhv}).

Trees can be represented as vectors of edge weights on the set of splits $\Sigma$.  Since the majority of the coordinates will have value $0$ (corresponding to splits not occurring in the tree), for clarity, we will sometimes suppress coordinates not under consideration and represent the tree by its non-zero edge weights only.

In other cases, we need to consider an embedding of treespace into $\R^N$, where $N = |\Sigma|$ is the number of possible splits being considered, and thus coordinates, on $n$ leaves.  As we are ignoring the splits that represent edges ending in leaves, $N$ is the number of partitions of the leaves into two parts, such that each part contains at least two leaves.  Thus $N = 2^{n-1} - n -1$.  
The order of the splits as coordinates in $\R^N$ is unimportant, but cannot change, so we assume some fixed ordering of the $N$ splits to correspond to the coordinates in $\R^N$.  For example, we can order the split sets lexicographically.  We now define a map from a vector of a subset of edges to this canonical ordering.  This is a specialization of Definition 5 in \cite{CLTcomplete}.

\begin{definition}
\label{def:jmath}
For any vector $V(E)$ of edge weights of a set of (interior) edges $E$, denote by $\jmath: V(E) \to \R^N$ the map that takes each coordinate value in $V(E)$ to the coordinate value in the vector in $\R^N$ representing the same split.  All other coordinate values in the vector in $\R^N$ are 0.
\end{definition}

We also define a projection function in treespace:

\begin{definition}
\label{def:coord_map}
For any tree $T \in \Tn$, and any set of compatible splits $E$, let $P_E(T)$ be the orthogonal projection of tree $T$ onto the orthant $\Or(E)$.  That is, let $P_E(T)$ be the tree containing only those edges in $E(T) \cap E$ with their weights as in $T$, or alternatively, the coordinate vector corresponding to this tree.
\end{definition}

\subsubsection{Geodesics in BHV Treespace} \label{s:geodesics}
Billera, Holmes, and Vogtmann \cite{bhv} defined a metric, which we call the \emph{BHV} or \emph{geodesic distance}, on this treespace as follows.  If two vectors representing trees are in the same orthant (that is, have the same non-zero coordinate values), then the distance between them is the Euclidean distance between them in $\R^N$.  If two trees are in different orthants, then the distance between them is the length of the shortest path between them, where the length of a path is the sum of the Euclidean lengths of the restriction of the path to each orthant that it traverses.  Billera et al.~\cite{bhv} showed that their treespace is globally non-positively curved \cite{bridson1999}, which implies that such shortest paths, or \emph{geodesics}, are unique.

We define 
$$
    ||T|| := \sqrt{ \sum_{e \in T} |e|_T^2 }
$$
to be the distance of tree $T$ to the origin.  Similarly, for a subset of edges $E$ in tree $T$, we define 
$$
||E||_{T} := \sqrt{\sum_{e \in E} |e|_T^2}.
$$

Owen and Provan \cite{Owen2011} gave a polynomial time algorithm for computing the geodesic, building on work characterizing the geodesic \cite{owen2008}.  It relies on the concept of {\em support}, which is a combinatorial condition on the orthants containing the geodesic.  We will use a slightly more general definition of support, following Barden and Le \cite{bardenTn}:

\begin{definition}
Let $T_1$ and $T_2$ be two trees in $\Tn$.  A \emph{support} is a pair of partitions $(\cal{A},\cal{B})$ where ${\cal A} = (A_0, A_1, ..., A_k)$ is a partition of $\left( E(T_1) \cup C(T_1)\right) \cap E(T_2)$ and ${\cal B} = (B_0, B_1, ..., B_k)$ is a partition of $\left( E(T_2) \cup C(T_2) \right) \cap E(T_1)$ such that:
\begin{enumerate}
\item $A_0 = B_0$ contain all edges corresponding to splits that are shared by the two trees or exist in one tree and are compatible with the other tree.  That is, 
\begin{align*}
    A_0 = B_0 = (E(T_1) \cap E(T_2)) 
     \cup  \left( E(T_1) \cap C(T_2) \right)
     \cup \left( C(T_1) \cap E(T_2) \right).
\end{align*}
\item $A_i$ is compatible with $B_j$ for all $1 \leq j < i \leq k$.
\end{enumerate}
A pair $(A_i, B_i)$ for $0 \leq i \leq k$ is called a \emph{support pair}.
\end{definition}

The shortest path, or geodesic, between two trees, $T_1$ and $T_2$, in treespace $\Tn$ is characterized by the following four properties \cite{Owen2011}. 

\begin{theorem}[{{\cite[Theorems 2.4 and 2.5]{Owen2011}}}]
\label{th:geo_characterization}
Let $T_1$ and $T_2$ be two trees in $\Tn$.  Then the support $(\cal{A},\cal{B})$ corresponds to the geodesic between $T_1$ and $T_2$ if and only if the following four properties hold:

\begin{itemize}
\item[P0:]  $A_0 = B_0$ contain all edges corresponding to splits that are shared by the two trees or exist in one tree and are compatible with the other tree.  That is, 
$A_0 = B_0 = \left( E(T_1) \cap E(T_2) \right)  \cup \left( E(T_1) \cap C(T_2) \right) \cup \left( C(T_1) \cap E(T_2) \right)$.
\item[P1:]  $A_i$ is compatible with $B_j$ for all $1 \leq j < i \leq k$.
\item[P2:] $\frac{||A_1||}{||B_1||} \leq \frac{||A_2||}{||B_2||} \leq \cdots \leq \frac{||A_k||}{||B_k||}$.
\item[P3:]  For every $(A_i, B_i)$ and non-trivial partitions $C_1 \cup C_2 = A_i$ and $D_1 \cup D_2 = B_i$ such that $C_2 \cup D_1$ are compatible, then $\frac{||C_1||}{||D_1||} > \frac{||C_2||}{||D_2||}$ holds.
\end{itemize}

Furthermore, if this support corresponds to the unique geodesic $\gamma = \{ \gamma(\lambda): 0 \leq \lambda \leq 1\}$, then it has segments:

\begin{equation}
\gamma_i = 
\begin{cases}
\left[\gamma(\lambda): 0 \leq \lambda < \frac{||A_{1}||}{||A_{1}|| + ||B_{1}||} \right], & \text{if}\ i = 0 \\

\left[\gamma(\lambda): \frac{||A_i||}{||A_i|| + ||B_i||} \leq \lambda < \frac{||A_{i+1}||}{||A_{i+1}|| + ||B_{i+1}||} \right], & \text{if}\ 1 \leq i < k  \\

\left[\gamma(\lambda): \frac{||A_i||}{||A_i|| + ||B_i||} \leq \lambda < 1 \right], & \text{if}\ i = k \\
\end{cases}
\end{equation}

where $\gamma_i$ is in the orthant $\Or(B_1 \cup \cdots B_{i-1} \cup A_i \cdots \cup A_k)$,
and the edge weights of $\gamma(\lambda)$ are 

\begin{equation}
|e|_{\gamma(\lambda)} = 
\begin{cases}
\frac{(1-\lambda)||A_j|| -\lambda||B_j||}{||A_j||}|e|_{T_1} & \text{if}\ e \in A_j \\

\frac{\lambda ||B_j|| - (1 - \lambda)||A_j||}{||B_j||}|e|_{T_2} & \text{if}\ e \in B_j \\

(1-\lambda)|e|_{T_1} + \lambda|e|_{T_2} & \text{if}\ e \in A_0 = B_0.
\end{cases}
\end{equation}

The length of the geodesic is 
$$ d(T_1, T_2) = \sqrt{ \sum_{i=1}^k (||A_i|| + ||B_i||)^2 + \sum_{e \in A_0 = B_0} (|e|_{T_1} - |e|_{T_2})^2 }. $$

\end{theorem}

\subsection{Fr\'echet Mean}
\label{s:mean_intro}

The Fr\'echet mean is a geometric center that generalizes the characterization of the Euclidean mean as the point minimizing the sum of squared distances to the input points \cite{frechet1948}. More precisely, if $\inputT = \{T_1, ..., T_r\}$ is the set of input trees, then the Fr\'echet mean minimizes the Fr\'echet function:

\begin{align} \label{eq:Frechet_mean}
f(T) = \sum_{i=1}^r d(T, T_i)^2
\end{align}
where $d$ is the BHV distance.  Sturm \cite{sturm2003} showed that the Fr\'echet mean is unique on globally non-postivively curved spaces, such as the BHV treespace \cite{bhv}.  Ba\v{c}\'ak \cite{bacakMeans} and Miller {\em et al.}~\cite{miller2015} independently adapted Sturm's  Law of Large Numbers \cite{sturm2003} for global non-positively curved spaces to give an iterative approximation algorithm for computing the Fr\'echet mean on treespace. 
Skwerer \cite{skwererMeans} gave a decomposition of the derivative of the Fr\'echet function that can be combined with optimization techniques to give an alternative algorithm for computing the Fr\'echet mean.  However, it is still an open question of whether the Fr\'echet mean can be computed in polynomial time.  Skwerer \cite{skwererMeans} notes that ``an indicator this problem is not NP-complete is randomized split-proximal point algorithms produce sequences of points with expected distances to the Fr\'echet mean converging to zero at a linear rate, and no approximation methods with such a rate of convergence exists for NP-complete optimization problems." 

We will need the following property of the mean, which has been expanded from the original version to include splits that are compatible with a tree but do not have positive weight in it:
\begin{prop}[{\cite[Lemma~5.1]{miller2015}}]
\label{prop:edges_common_to_all_trees}
Every split in the mean tree is a split in some input tree.  Furthermore, if a split appears some input trees, and is compatible with all input trees that it does not appear in, then that split must also appear in the mean tree.
\end{prop}


This property can be extended to give the weight of the common split in the mean tree, as well as show that computing the mean can be decomposed along these common splits.  While this extension was previously known, to our knowledge the following proof is the first place that it has been written down.  Let $s = Y|Z$ be a split in tree $T$ which has corresponding edge $e = \{u,v\}$ where $u$ is the vertex connecting $e$ to the subtree with leaves $Y$ and $v$ is the vertex connecting $e$ to the subtree with leaves $Z$.  If $e$ is a pendant edge, then $u = Y$ or $v = Z$.  Let $T^Y$ be the subtree induced by leaves $Y \cup u$ (that is, $u$ will become a leaf with a zero length pendant edge in the new subtree) and let $T^Z$ be the subtree induced by leaves $Z \cup v$ (that is, $v$ will become a leaf with a zero length pendant edge in the new subtree).  See Figure~\ref{fig:commonSplitMeanTree}.

\begin{figure}
    \centering
    \begin{tabular}{ccc}
    \includegraphics[scale = 0.4]{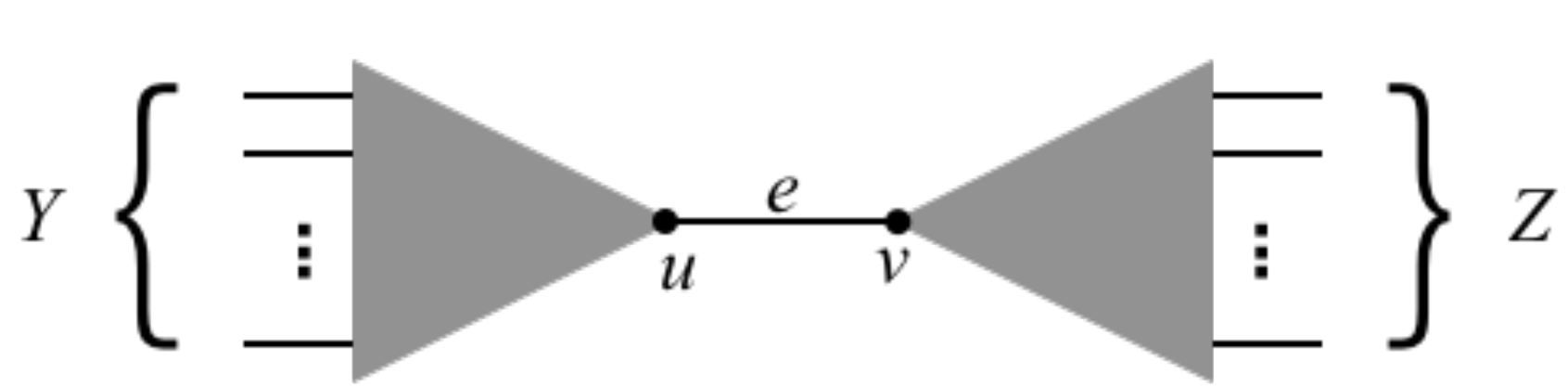} &
    \includegraphics[scale = 0.4]{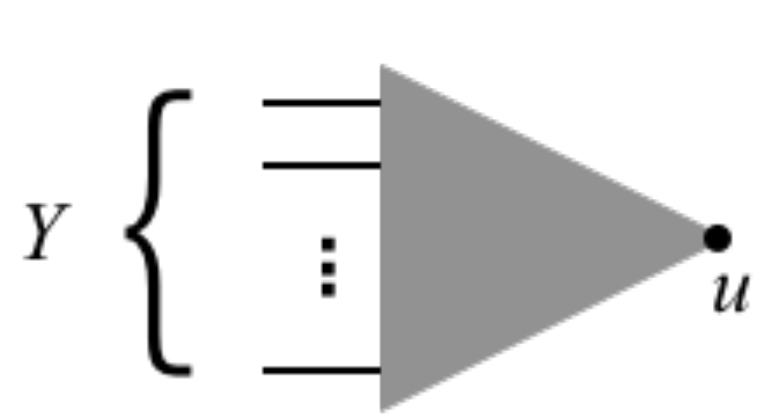} &
    \includegraphics[scale = 0.4]{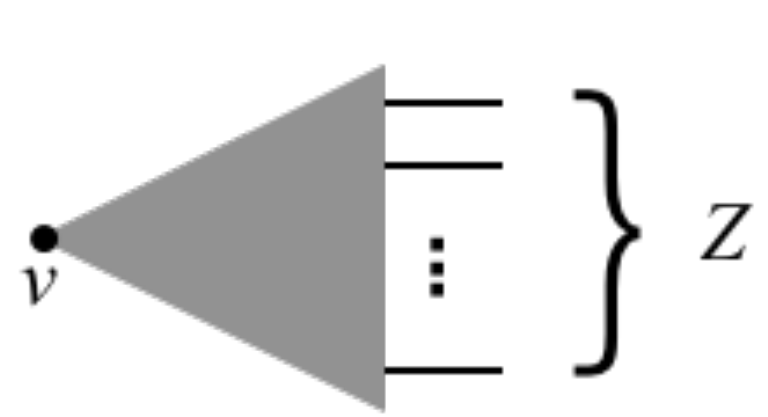} \\
    (a) & (b) & (c)
    \end{tabular}
    \caption{a) Tree $T$ with common split $Y|Z$ corresponding to edge $e = \{u,v\}$.  Deleting edge $e$ splits tree $T$ into b) tree $T^Y$ with leaves $Y \cup u$ and c) tree $T^Z$ with leaves $Z \cup v$.}
    \label{fig:commonSplitMeanTree}
\end{figure}

\begin{lemma}
\label{lem:common_split}
Let $\inputT = \{T_1, ..., T_r\}$ be a set of trees in $\Tn$, with common  (both interior and pendant) splits $C$ and Fr\'echet mean tree $T$.  Then for each split $s = Y|Z \in C$ corresponding to edge $e = \{u,v\}$, the weight of $s$ in $T$ is $\frac{1}{r}\sum_{i = 1}^r |s|_{T_i}$.  Furthermore, for each such $s$, $T$ is composed of the mean of $\{T_1^Y, ..., T_r^Y\}$ and the mean of $\{T_1^Z, ..., T_r^Z\}$ joined by an edge with weight $|s|_T $ between leaves $u$ and $v$.
\end{lemma}

\begin{proof}

By \cite[Theorem 2.1]{owen2011computing}, which was originally proven by Vogtmann \cite{vogtmann2007geodesics}, we can re-write the BHV distance as $d(T,T_i) = \sqrt{d(T^Y,T_i^Y)^2 + d(T^Z,T_i^Z)^2 + (|e|_T - |e|_{T_i})^2}$ for all $1 \leq i \leq k$.  The distances $d(T^Y,T_i^Y)$ and $d(T^Z,T_i^Z)$ are taken in the BHV treespaces for trees with $|Y \cup u|$ and $|Z \cup v|$ leaves, respectively.  Plugging this distance expression into the Fr\'echet mean function gives:

\begin{align*}
   f(T) &= \sum_{i = 1}^r d(T, T_i)^2 \\
   &= \sum_{i=1}^r \left( \sqrt{d(T^Y,T_i^Y)^2 + d(T^Z,T_i^Z)^2 + (|s|_T - |s|_{T_i})^2} \right)^2 \\
   &= \sum_{i=1}^r d(T^Y,T_i^Y)^2 + \sum_{i=1}^r d(T^Z,T_i^Z)^2 + \sum_{i=1}^r (|s|_T - |s|_{T_i})^2
\end{align*}
We can minimize each of the three sums $\sum_{i=1}^r d(T^Y,T_i^Y)^2$, $\sum_{i=1}^r d(T^Z,T_i^Z)^2$, and $\sum_{i=1}^r (|s|_T - |s|_{T_i})^2$ independently since $T^Y$ and $T^Z$ are non-overlapping subtrees of $T$, connected by the edge corresponding to split $s$.  The first two sums $\sum_{i=1}^r d(T^Y,T_i^Y)^2$ and $\sum_{i=1}^r d(T^Z,T_i^Z)^2$ are minimized by the mean trees $T^Y$ and $T^Z$ of $\{T_1^Y, ..., T_r^Y\}$ and $\{T_1^Z, ..., T_r^Z\}$, respectively.  The expression $\sum_{i=1}^r (|s|_T - |s|_{T_i})^2$ is a least squares function of Euclidean distances, and thus is minimized by the Euclidean average $\frac{1}{r} \sum_{i=1}^r |s|_{T_i}$.
\end{proof}

\subsubsection{Stickiness of the Mean}

In Euclidean space, if the mean of a set of points in $\mathbb{R}^n$ is computed and then one of the input points is perturbed, the mean will always change position. It is not ``sticky.''  In BHV treespace and other spaces, the Fr\'echet mean is ``sticky" in certain situations, meaning perturbing an input tree will not change the mean tree.  Stickiness only occurs when the mean is on a lower-dimensional orthant. This phenomena was first reported for BHV treespace in \cite{miller2015}, and has been studied in conjunction with Central Limit Theorems for open books \cite{hotzStickyOpenBooks} and hyperbolic planar  singularities \cite{huckemannStickyHyperCone}, which generalize features of treespace, and for BHV treespace itself \cite{bardenT4, bardenTn, CLTcomplete}.   

We now give an example of a sticky mean. This example will be used later in the paper to provide some counter-examples related to our work.

\begin{figure}[t]
\begin{center}
\begin{tabular}{ccc}
 \includegraphics[height=2in]{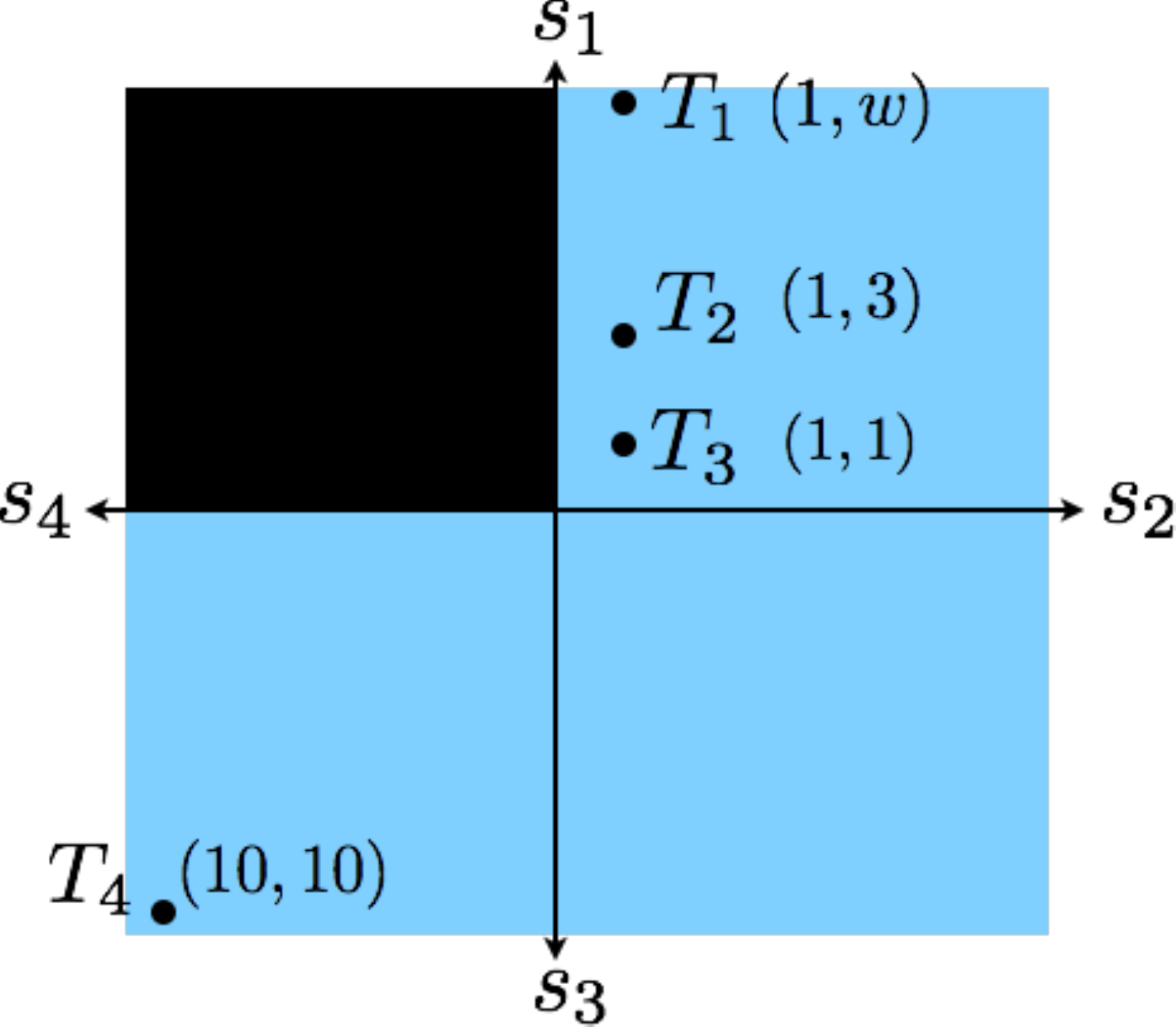}
 &\mbox{\hspace{.75in}}&
\includegraphics[scale=0.4]{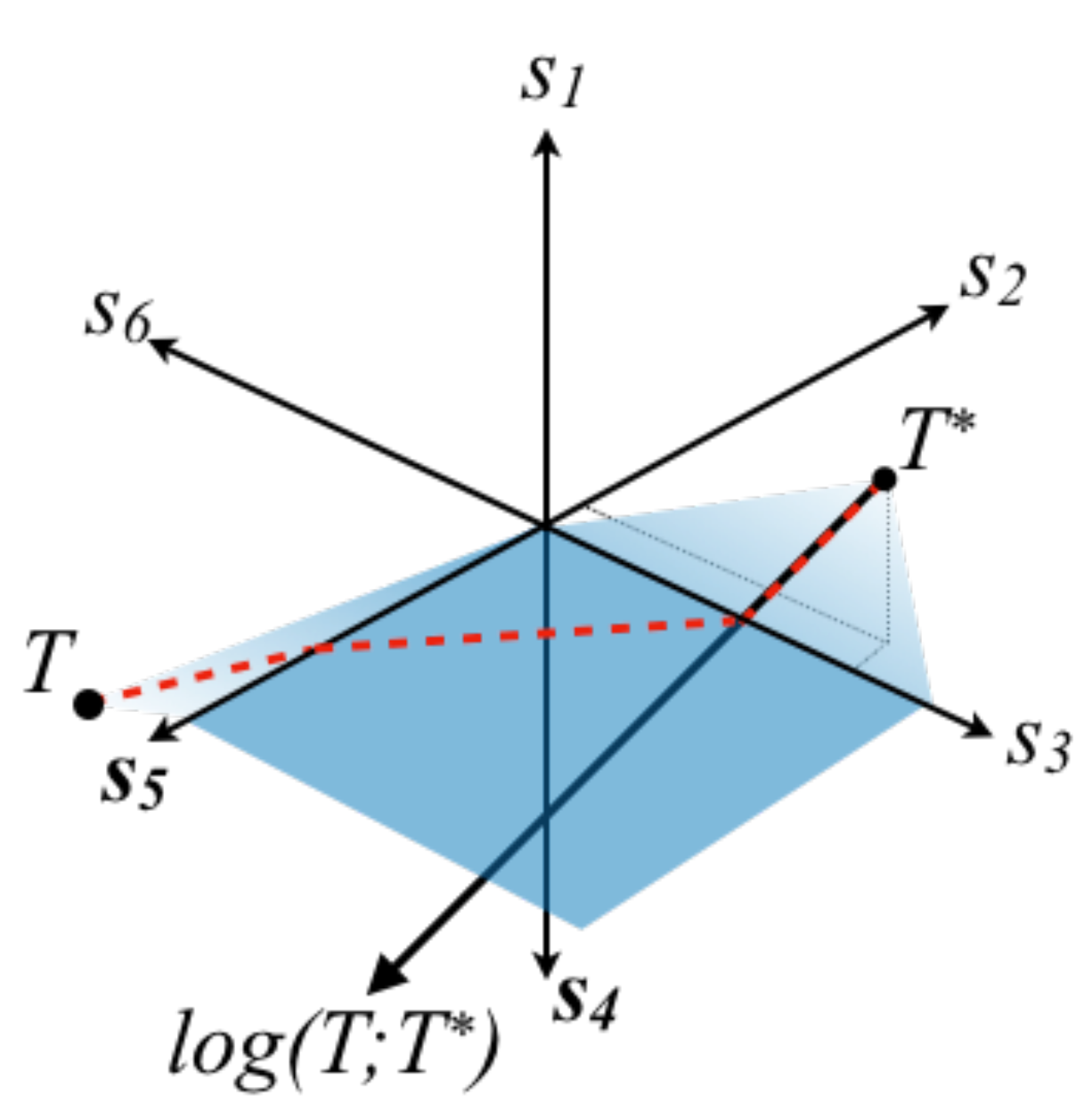}
\\
\\
(a) && (b)
\end{tabular}
\end{center}

    \caption{(a)  An example of three trees with the same topology on $5$ leaves, 
    $T_1, T_2, T_3$, and a fourth tree, $T_4$ on two incompatible splits.   
    The weight of the $s_2$ split in tree $T_1$ affects the topology of the mean tree of the $T_1$,$T_2$,$T_3$, and $T_4$.
    (b) The orthants $(s_1, s_2, s_3)$, $(s_3,s_5)$, and $s_4,s_5,s_6)$ are the only orthants in treespace.  The geodesic from tree $T^*$ to $T$ is shown as a dashed line.  The log map of $T$ based at $T^*$, $log(T;T^*)$, is shown as a vector.  It coincides with the geodesic in the first orthant $s_1,s_2,S_3$.   
    }
    \label{fig:ex_shivam_ella}
    \label{fig:logmap}
\end{figure}

\begin{example}  \label{ex:shivam_and_ella}
Consider the four input trees shown in Figure~\ref{fig:ex_shivam_ella}b, where the internal splits of the trees $T_1$, $T_2$, and $T_3$ are $s_1$ and $s_2$ while the tree $T_4$ has internal splits $s_3$ and $s_4$. Splits $s_1$ and $s_4$ are not compatible, so they are not contained in a single orthant and we indicate that the region is not an orthant by shading it in black in Figure~\ref{fig:ex_shivam_ella}a. By Proposition~\ref{prop:edges_common_to_all_trees}, the mean of $\{T_1, T_2, T_3, T_4\}$ lies in one of the other three orthants.

In Figure~\ref{fig:ex_shivam_ella}a, the three orthants have been embedded in the Euclidean plane, with the origins coinciding.  If the trees were actually points in this plane, then their Euclidean mean would be at $(\frac{1 + 1 + 1 - 10}{4},\frac{w + 3 + 1 -10}{4}) = (\frac{-7}{4},\frac{w - 6}{4})$, which is not in any of the three allowable orthants when $w > 6$.  For the rest of the example, we consider the case where $w > 6$ and is fixed.  For any point in the $(s_2,s_3)$-orthant, the shortest paths from that point to $T_1$, $T_2$, $T_3$, $T_4$ coincide in treespace and the Euclidean embedding.  Thus, if the mean tree were in the $(s_2,s_3)$-orthant, it would also be the Euclidean mean of the embedded points.  However, since the Euclidean mean is not in this orthant, the mean tree cannot be either.   

Next, consider the case where the mean tree is in the $(s_3,s_4)$-orthant.  The geodesics from $T_4$ to $T_1$, $T_2$, and $T_3$ all lie on the line from $T_4$ to the origin, which implies the mean tree must also lie on this line if it is in the $(s_3, s_4)$-orthant.  Letting $T$ be the tree at $(x,x)$, for $x >0$, in the $(s_3, s_4)$-orthant, the Fr\'echet function (Equation \ref{eq:Frechet_mean}) becomes
\begin{align*}
   f(T) &= \sum_{i=1}^4 d(T, T_i)^2 \\
   &= (||(x,x)|| + ||(1,w)||)^2 + (||(x,x)|| + ||(1,3)||)^2 + (||(x,x)|| + ||(1,1)||)^2 + ||(10,10)-(x,x)||^2 \\
   &= (\sqrt{2}x + \sqrt{1 + w^2})^2 + (\sqrt{2}x + \sqrt{10})^2 + (\sqrt{2}x + \sqrt{2})^2 + 2(10-x)^2.
\end{align*}


Therefore, our Fr\'echet function is a function in $x$, $f(x) =(\sqrt{2}x + \sqrt{1 + w^2})^2 + (\sqrt{2}x + \sqrt{10})^2 + (\sqrt{2}x + \sqrt{2})^2 + 2(10-x)^2$, that should be minimized to find the mean.  
%
%
Using Sage \cite{sagemath}, we determine that the restriction that $x >0$ implies that $w < 9.51$.  For example, when $w = 9$, $f(x)$ is minimized by $x \sim 0.0902$, and thus the mean is at $(0.0902, 0.0902)$ in the $(s_3,s_4)$-orthant.  Note that we have not strictly proven here that this is the mean, but it can be verified by applying Theorem~\ref{th:mean_characterization}.

Next, consider the case where the mean tree is in the $(s_1,s_2)$-orthant.  Then the geodesics from the mean to $T_1$, $T_2$, and $T_3$ are straight lines and remain the same in the Euclidean embedding.  Furthermore, the mean must lie on or above the line from $T_3$ to $T_4$, otherwise we could project it onto this line to get a smaller Fr\'echet function. 
Therefore, the geodesic from the mean to $T_4$ always passes through the origin.  So computing the mean in treespace is equivalent to computing the mean in Euclidean space where $T_1$, $T_2$, and $T_3$ have the same positions and $T_4$ is replaced by a point at distance $2\sqrt{10}$ from the origin in the $(s_3,s_4)$-orthant on the line extending from the treespace mean, through the origin, into the $(s_3,s_4)$-orthant.  Call this new point $T_4'$.  To compute this equivalent Euclidean mean we will find the mean $\mu_{123}$ of $T_1$, $T_2$, and $T_3$, and then take the weighted mean of $\mu_{123}$ and $T_4'$.  Then $$\mu_{123} = \left( \frac{1 + 1 + 1}{3},\frac{w + 3 + 1}{3} \right) = \left( 1, \frac{4+w}{3} \right).$$


We now need to take the weighted mean of $\mu_{123}$ and $T_4'$.  Equivalently, we can compute the weighted 1-dimensional Euclidean mean $m$ of a point at $||(1,\frac{4+w}{3})|| = \sqrt{1 + \frac{(4 + w)^2}{9}}$ with weight 3, and a point at $-10\sqrt{2}$ with weight 1.  
Then

\begin{align*}
    m &= \frac{(3) \left( \sqrt{1 + \frac{(4 + w)^2}{9}}\right) +(-10\sqrt{2})}{4}
    = \frac{\sqrt{9 + (4 + w)^2} -10\sqrt{2}}{4}.
\end{align*}

The mean $m$ must be positive to lie in the $(s_1, s_2)$-orthant, implying $w > \sqrt{191} - 4 \sim 9.82...$.  Again, we can verify this is indeed the mean by applying Theorem~\ref{th:mean_characterization}.



Summarizing, if $6 <w < 9.51$, then the Fr\'echet mean of $T_1$, $T_2$, $T_3$, and $T_4$ is in the $(s_3,s_4)$-orthant.  If $w > 9.82$, then the Fr\'echet mean is in the $(s_1, s_2)$-orthant.  If $w$ is in between these values, then since it cannot be in the $(s_2,s_3)$-orthant, the Fr\'echet mean is at the origin, demonstrating stickiness.

\end{example}

\subsection{The log map and translated log map}
\label{sec:logmap}

A key tool in differential geometry is the tangent space at a point on a manifold, which contains the directions of all tangent lines passing through that point.  The tangent space can be generalized to a tangent cone at manifold singularities. Vectors can be mapped from the tangent cone to the manifold by the exponential map, and from the manifold to the tangent cone by the logarithm (or log) map, the inverse of the exponential map.  

The tangent spaces, tangent cones, and log maps were defined for BHV treespace in \cite{CLTcomplete,bardenT4, bardenTn}, and follow the definitions in general CAT(0) spaces \cite[Definition 3.18]{bridson1999}.  In treespace, intuitively, the tangent cone at a tree $T^*$ corresponds to the cone formed by taking all vectors in the neighborhood of $T^*$ that start at $T^*$, and extending these vectors into rays.
  
For example, if $T^*$ is a  binary tree, then $T^*$ is in the interior of a top-dimensional orthant.  Therefore, the tangent cone is a tangent space, a Euclidean space of the same dimension as the orthant.  See Figure~\ref{fig:3flaps}.  If $T^*$ is on an axis in $\mathcal{T}_5$, like $T_4$ in Figure~\ref{fig:3flaps}, then the tangent cone at $T^*$ is three half planes meeting at a shared axis.  
If $T^*$ is at the origin of $\Tn$, then the tangent cone at $T^*$ looks like $\Tn$.  Intuitively, the log map at $T^*$ maps a tree $T$ in $\Tn$ onto a point in the tangent space at $T^*$ that is the same distance and starting direction from $T^*$ as in treespace.  That is, the log map ``unfolds" the geodesic from $T^*$ to $T$ into the tangent cone to treespace at $T^*$.  See Figure~\ref{fig:logmap}b.

We now give the formal definitions of the tangent cone and log map at any tree $T^* \in \Tn$.

\begin{definition}
For any tree $T^* \in \Tn$, the \emph{tangent cone} to $\Tn$ at $T^*$ consists of all initial tangent vectors to smooth curves starting from $T^*$, where smoothness may be only one-sided at $T^*$.
\end{definition}

\begin{definition}
For any tree $T^* \in \Tn$, define the \emph{log map} at $T^*$ to be the map from $\Tn$ to the tangent cone at $T^*$ given by $\log(T;T^*) = d(T^*,T)\textbf{v}(T)$, where $T$ is any tree in $\Tn$, $d(T^*,T)$ is the BHV distance between $T^*$ and $T$ and $\textbf{v}(T)$ is the unit vector in the direction that the geodesic from $T^*$ to $T$ leaves $T^*$.
\end{definition}

Since the tangent cone at $T^*$ contains rays in all possible directions from $T^*$, the log map is well-defined.  Following \cite{CLTcomplete}, we will work with a translation of the log map, called the \emph{translated log map}, that translates the log map by $T^*$ so that the origin of the log map matches the origin in treespace.  This is possible since the tangent cones at all points in some (not necessarily top-dimensional) orthant are parallel (in the differential geometry sense), and thus can be parallel translated to the origin.

\begin{definition}
For trees $T, T^* \in \Tn$, define the \emph{translated log map} to be $\Phi(T;T^*) = log(T;T^*) + T^*$.
\end{definition}

By specializing \cite[Theorem 1]{CLTcomplete} to the treespace case, and recalling that $P_E(T)$ is the orthogonal projection of tree $T$ onto the orthant $\Or(E)$, we have an expression for the coordinates of the translated log map.  

\begin{theorem}[{\cite[Theorem~1]{CLTcomplete}}]
\label{th:log_coords}
For trees $T, T^* \in \Tn$, let $({\cal A},{\cal B})$, where ${\cal A} = (A_0, ..., A_k)$ and ${\cal B} = (B_0, ..., B_k)$, be the support of the geodesic from tree $T^*$ to $T$.  Then the translated log map $\Phi(T;T^*)$ at $T^*$ is
\begin{align*}
\Phi(T;T^*) = \jmath \left ( P_{B_0}(T), -\frac{||P_{B_1}(T)||}{||P_{A_1}(T^*)||} P_{A_1}(T^*), ..., -\frac{||P_{B_k}(T)||}{||P_{A_k}(T^*)||} P_{A_k}(T^*) \right ),
\end{align*}
where $\jmath$ is the map give in Definition~\ref{def:jmath}.
Alternatively, we can write this as
\begin{align}
\Phi(T;T^*) = \jmath \left ( B_0, -\frac{||B_1||}{||A_1||}A_1, ..., -\frac{||B_k||}{||A_k||} A_k \right ).
\end{align}
\end{theorem}

Barden and Le  gave a theorem \cite[Theorem 3]{CLTcomplete} characterizing when a point in a CAT(0) orthant space is the Fr\'echet mean of a given distribution.  Before we specialize this theorem to treespace, we need to introduce their idea of a directional limit of the translated log map.

A tree $T^*$ in a lower dimensional orthant is on the boundary of multiple higher dimensional orthants.   It is often useful to consider $T^*$ as belonging to one of these orthants, and to construct a translated log map at $T^*$ from this perspective.  For example, in Figure~\ref{fig:3flaps} , $T_4$ is on the axis between the three orthants $\Or_1$, $\Or_2$, and $\Or_3$.  By considering $T_4$ as part of $\Or_1$, $\Or_2$, and $\Or_3$ in turn, we get the three different ways to flatten the three orthants into the plane (Figures~\ref{fig:3flaps}a, b, and c respectively).  Taking the translated log map at a point in the interior of that orthant gives a similar flattening.

\begin{figure}
    \centering
    \begin{tabular}{ccc}
    &\includegraphics[scale=0.4]{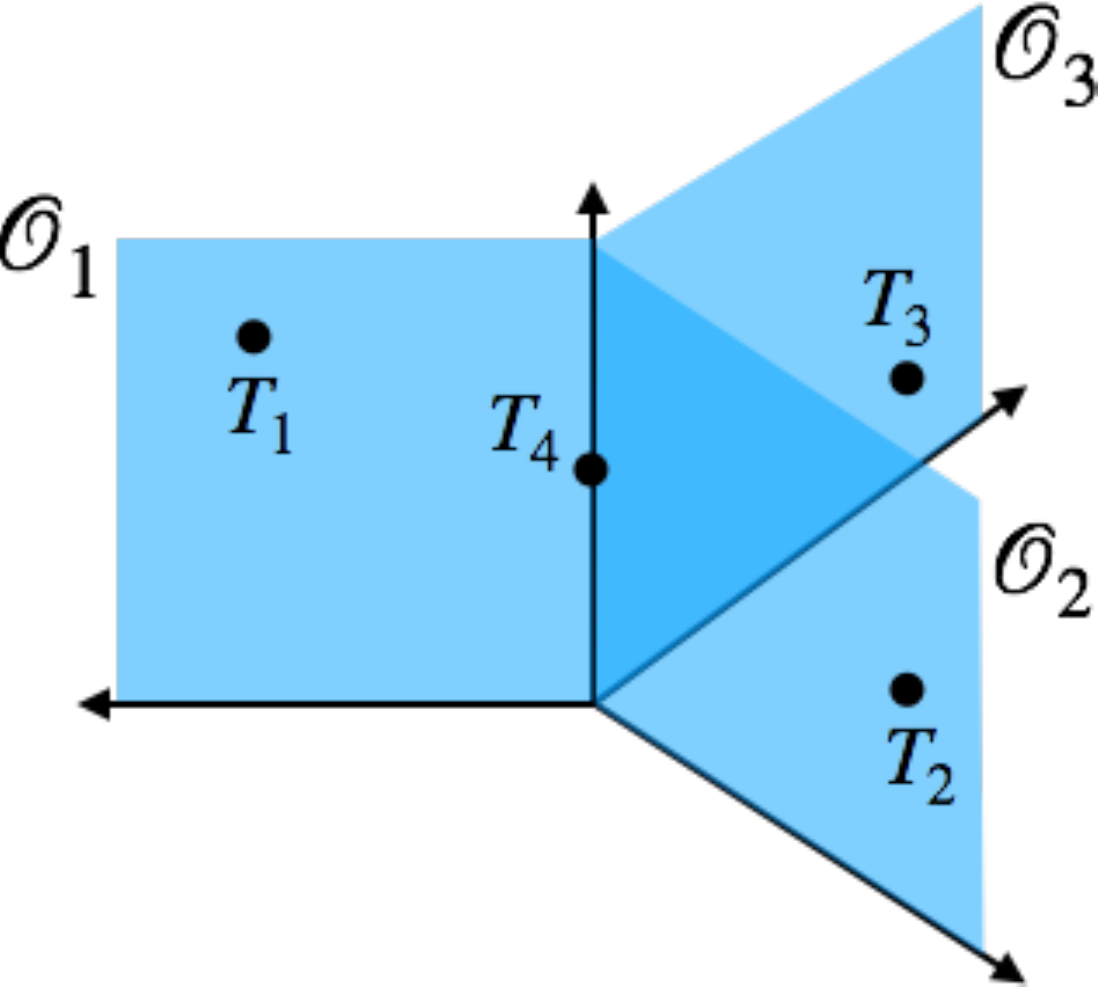} &\\
    & (a) \\
    \\
    \includegraphics[scale=0.4]{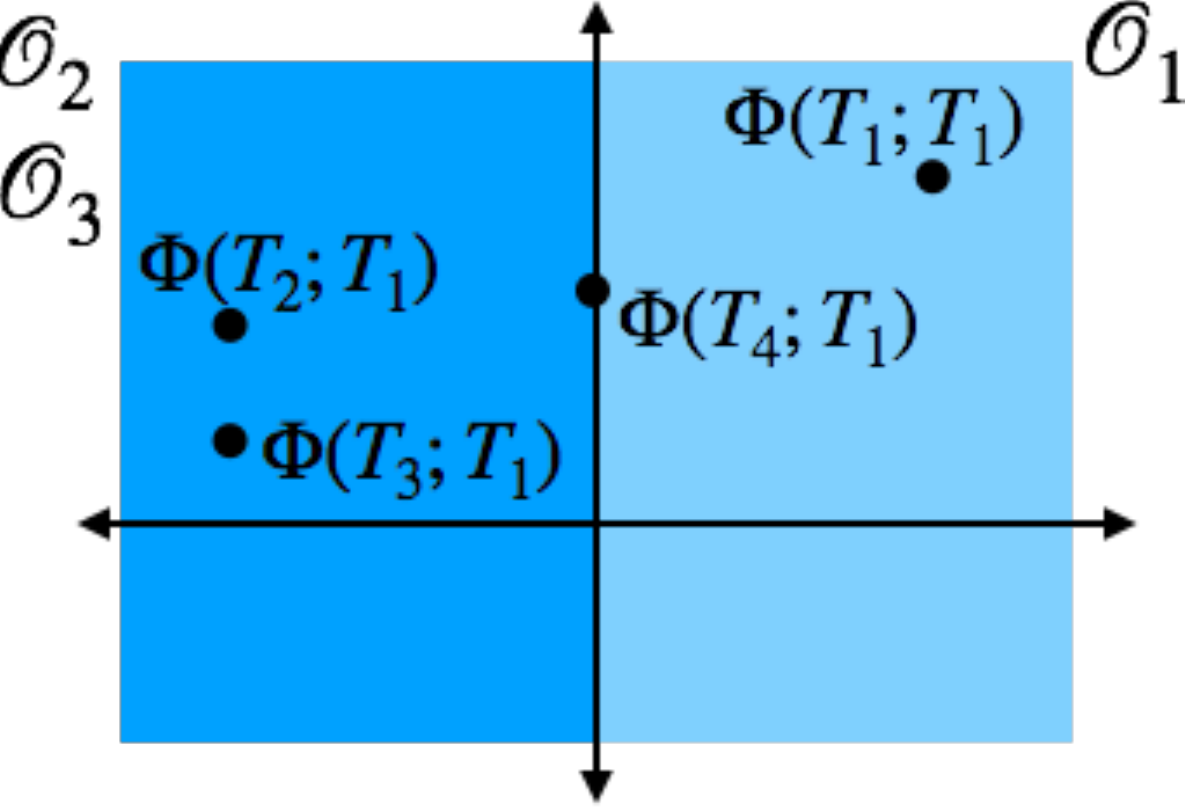} &
    \includegraphics[scale = 0.4]{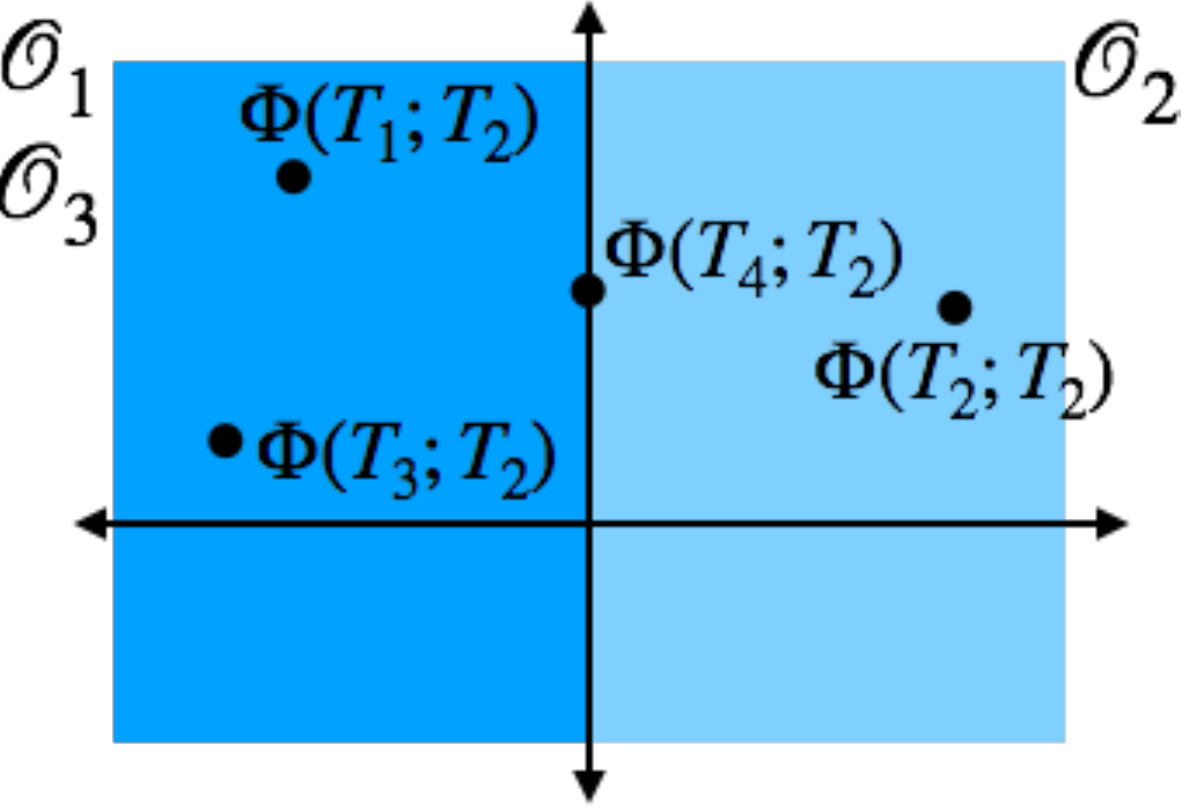} &
    \includegraphics[scale = 0.4]{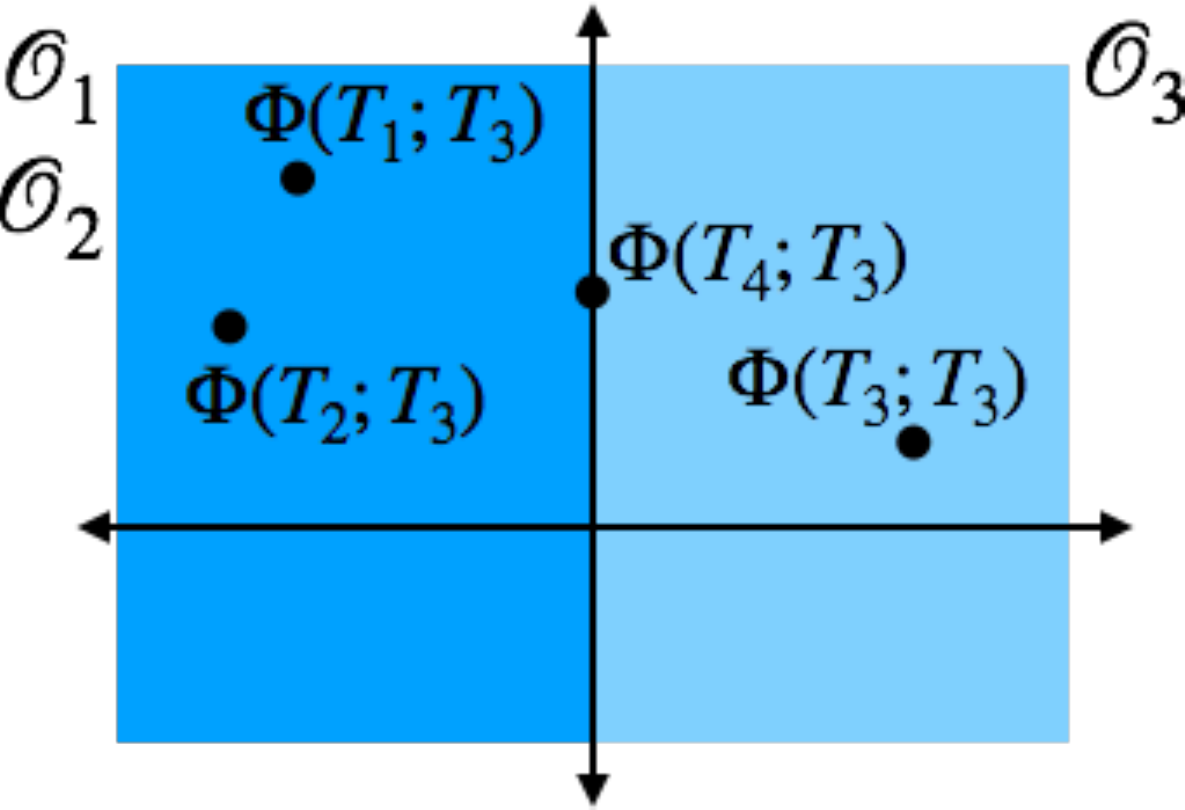} \\
    \\
    (b) & (c) & (d)
    \end{tabular}
    \caption{(a) Three trees, $T_1$, $T_2$, and $T_3$, are in three different orthants that share an axis, and one tree, $T_4$, lies on this axis.  The translated log maps based at $T_1$, $T_2$, and $T_3$ are shown in (b), (c), and (d), respectively.  In each log map, the two orthants not containing the base tree have been identified with each other, and the shared axis is extended in the negative direction. The tangent spaces shown in (b), (c), and (d) are equivalent to the directional limit of the translated log map based at $T_4$ when approached from the direction of the $\Or_1$, $\Or_2$, and $\Or_3$ orthants, respectively.
    }
    \label{fig:3flaps}
\end{figure}

We now formalize this description.  For tree $T^*$ in a lower dimensional orthant $\Or$, let $w$ be a vector in the tangent cone at $T^*$.  For any $\lambda > 0$, let $T^*(\lambda,w)$ represent the tree $\lambda ||w||$ along some geodesic starting at $T^*$ with initial tangent vector $w$.

\begin{definition}
\label{def:directional_limit}
Let $T^*$ be a tree in $\Tn$ with splits $E = E(T^*)$, let $T$ be any other tree in $\Tn$, and let $w$ be any vector in the tangent cone at $T^*$.  Then the \emph{directional limit} is defined as $\lim_{\lambda \to 0+} \Phi(T;T^*(\lambda,w))$.

\end{definition}

Specializing \cite[Theorem 2]{CLTcomplete} to the treespace case, we see that the directional limit exists and can be computed.

%
%
\begin{theorem}[{\cite[Theorem~2]{CLTcomplete}}]
\label{th:directional_limit}
Let $T^*$ be a tree in $\Tn$ with splits $E = E(T^*)$, and let $T$ be any other tree in $\Tn$.

\begin{enumerate}

\item If $w$ is a vector in the tangent cone at $T^*$ with non-zero values only in the coordinates corresponding to splits $E$, then the directional limit $\lim_{\lambda \to 0+} \Phi(T;T^*(\lambda,w)) = \Phi(T;T^*)$.

\item If $F$ is any set of splits such that $E \cap F = \emptyset$ and $E \cup F$ is mutually compatible set of splits, and if $w_{E \cup F}$ is a vector in the tangent cone at $T^*$ with the only non-zero coordinate values corresponding to splits $E$ and $F$, with the values corresponding to splits $F$ being strictly positive, then the limit $\Psi(T,w_{E \cup F};T^*) = \lim_{\lambda \to 0+} \Phi(T;T^*(\lambda,w_{E \cup F}))$ exists.  Furthermore, there exists some $\epsilon > 0$ such that the geodesic from $T^*(\lambda,w_{E \cup F})$ to $T$ has the same geodesic support for all $\lambda \leq \epsilon$.  Let this support be $({\cal A},{\cal B}) = ((A_0, ..., A_k),(B_0, ..., B_{k}))$. Then
\begin{align*}
\Psi(T,w_{E \cup F};T^*) = \jmath \left ( P_{B_0}(T),-\frac{||P_{B_1}(T)||}{||W_1||}W_1, ..., -\frac{||P_{B_k}(T)||}{||W_k||}W_k \right ), 
\end{align*}
where $W_i = P_{A_i \cap E}(T^*)$, unless $P_{A_i \cap E}(T^*) = 0$, in which case $W_i = P_{A_i \cap F}(w_{E \cup F})$, and $\jmath$ is the linear transformation defined in Definition~\ref{def:jmath}.
\end{enumerate}

\end{theorem}


Note that we are abusing notation by writing $P_{A_i \cap F}(w_{E \cup F})$ in the second part of the above theorem, since the function input should be a tree.  However, as we are projecting onto $A_i \cap F$ and $w_{E \cup F}$ has only non-negative values in the coordinates corresponding to the splits in $F$, the projection will only have positive values in a subset of compatible splits $F$, corresponding to a tree or point in the BHV treespace. 

For more compact notation, we will represent the projections of $\Phi(T;T^*)$ and $\Psi(T,w;T^*)$ onto a set $S$ of compatible splits by $\Phi_S(T;T^*) = P_S(\Phi(T;T^*))$ and $\Psi_S(T,w;T^*) = P_S(\Psi(T,w;T^*))$.  In these cases, the projection is onto the tangent plane $\R^{|S|}$.  We will need the expression for the coordinates of $\Psi_{E \cup F} (T,w_{E \cup F};T^*)$ which is given in Equation 20 in \cite{CLTcomplete} and follows the notation of the above theorem (Theorem~\ref{th:directional_limit}):
\begin{align}
\label{eq:projection_of_Psi_coords_with_w_E_cup_F}
\Psi_{E \cup F}(T,w_{E \cup F};T^*) = \jmath \left ( P_{B_0 \cap (E \cup F)}(T),-\frac{||P_{B_1}(T)||}{||W_1||}W_1, ..., -\frac{||P_{B_k}(T)||}{||W_k||}W_k \right ).
\end{align}

Next we specialize  \cite[Theorem 3]{CLTcomplete}, which gives a characterization of the Fr\'echet mean, to the treespace case. 
We will use this theorem to give necessary and sufficient conditions for what splits are part of the mean tree.


\begin{theorem}[{\cite[Theorem~3]{CLTcomplete}}]
\label{th:mean_characterization}
Let $\inputT = \{ T_1, ..., T_r\}$ be a set of trees in $\Tn$.  Suppose the tree $\mu$ has the non-zero splits $S = \{s_1, ..., s_m \}$, where $m \leq n-3$.  Then tree $\mu$ is the Fr\'echet mean of trees $\inputT$ if and only if:

\begin{enumerate}[(i)]
\item for any set of splits $F$, such that $S \cap F = \emptyset$ and $S \cup F$ is a set of mutually compatible splits, and for any unit vector $w_F$ in the tangent cone at $\mu$ that has only non-zero coordinates for splits in $F$ and these coordinates are positive,  then
\begin{align*}
\left \langle w_F, \sum_{T \in \inputT} \Psi_{S \cup F}(T,w_F; \mu) \right \rangle \leq 0
\end{align*}

\item
\begin{align*}
	\Phi(\mu;\mu) &= \frac{1}{r}\sum_{T \in \inputT}  \Phi_{S}(T;\mu) \\
\end{align*}

\end{enumerate}

\end{theorem}


\section{Necessary and Sufficient Conditions for Splits in the Mean}

In this section, we give some necessary conditions and some sufficient conditions for a split to be in the mean tree.  These conditions take the form of inequalities on split weights in the input trees.  We conclude the section with some examples showing that these conditions are not tight, and thus do not give a characterization for when a split is in the mean. Nevertheless, these conditions can be used to improve the computation of the mean, as shown in the following section. 
Throughout this section, we will assume that we are trying to find the Fr\'echet mean of the trees $\inputT = \{T_1, T_2, ..., T_r\}$ in $\Tn$.  Our basic approach is to take the log map of the input trees at their mean, which we assume, but do not actually know.  However, the coordinates of the images of these log maps must satisfy the two equations given in Theorem~\ref{th:mean_characterization}, which characterizes the mean.  Using properties of the geodesic, we can re-write the resulting expressions to eliminate all references to the assumed splits and weights of the mean tree, yielding the desired inequalities.  



The following lemma is a basic inequality on the weights of splits appearing in the mean tree, and will be used as a starting point for Lemma~\ref{lem:split_sum_incompatible} and Theorem~\ref{th:mean_implies}.  Recall that $X(T)$ is the set of splits incompatible with tree $T$.  This lemma states that for each split in the mean tree, the sum of the weights of that split in the input trees is greater than the sum of adjusted weights of certain incompatible splits in the input trees.  This lemma is derived by noting that the average of the coordinates of the images of the input trees under the log map at their mean must all be positive, since these averages are equal to the mean split weights, which are positive by definition, by Theorem~\ref{th:mean_characterization}, part (ii).

\begin{lemma}
\label{lem:coord_sums_are_positive}
Let $\mu$ be the mean tree of the input trees $\inputT = \{T_1, T_2, ..., T_r\}$, and let $S = \{s_1, s_2, ..., s_m\}$ be the splits with positive weight in $\mu$.  For each tree $T \in \inputT$, let $({\cal A}^T,{\cal B}^T)$ be the support of the geodesic from $\mu$ to $T$, with ${\cal A}^T$ and ${\cal B}^T$ both having $k(T) + 1$ partitions, namely ${\cal A}^T = (A_0^T, A_1^T,..., A_{k(T)}^T)$ and ${\cal B}^T = (B_0^T, B_1^T, ..., B_{k(T)}^T)$.  For each $1 \leq i \leq m$ and $T \in \inputT$, let $p(i,T)$ be the subscript of the partition of ${\cal A}^T$ containing $s_i$.  
Then:
\begin{align}
\sum_{\substack{T \in \inputT:\\s_i \in E(T)}} |s_i|_T &> \sum_{\substack{T \in \inputT:\\ s_i \in X(T)}} \frac{||B_{p(i,T)}^T||}{||A_{p(i,T)}^T||}|s_i|_{\mu}.
\end{align}
\end{lemma}

\begin{proof}
Consider $\sum_{T \in \inputT} \Phi_S(T; \mu)$ from the right side of Theorem~\ref{th:mean_characterization}, part (ii).  To find the coordinate values of $\Phi_S(T;\mu)$, we use Theorem~\ref{th:log_coords} to get $\Phi(T;\mu) = \jmath \left ( B_0, -\frac{||B_1^T||}{||A_1^T||}A_1^T, ..., -\frac{||B_k^T||}{||A_k^T||} A_k^T \right )$.  Projecting $\Phi(T;\mu)$ onto $S$ keeps only those coordinates corresponding to splits in the mean, namely  $s_1, s_2, ..., s_m$.  For each split $s_i \in S$, its coordinate value in $\Phi_S(T;\mu)$ is $|s_i|_{T}$ if $p(i,T) = 0$ and $-\frac{||B_{p(i,T)}^T||}{||A_{p(i,T)}^T||}|s_i|_{\mu}$ otherwise.  Note that if $s_i$ is compatible with $T$ but not a split in $T$, meaning $s_i \in C(T)$, then $s_i \in B_0$ by the definition of the support, but $|s_i|_T = 0$.

After summing over all trees in $\inputT$, for each $1 \leq i \leq m$, the corresponding coordinate value for split $s_i \in S$ in $\sum_{T \in \inputT} \Phi_S(T; \mu)$ is
\begin{align*}
&\sum_{T \in \inputT} \left ( \sum_{s_i \in E (T) \cup  C(T)} |s_i|_T + \sum_{s_i \in X(T)} -\frac{||B_{p(i,T)}^T||}{||A_{p(i,T)}^T||}|s_i|_{\mu} \right ) \\
&= \sum_{T \in \inputT} \left ( \sum_{s_i \in E (T)} |s_i|_T + \sum_{s_i \in X(T)} -\frac{||B_{p(i,T)}^T||}{||A_{p(i,T)}^T||}|s_i|_{\mu} \right )
\end{align*}
where the second line follows from $s_i \in C(T)$ implying $|s_i|_T = 0$ (since it is compatible but not part of $T$).

Theorem~\ref{th:mean_characterization}, part (ii), states $\Phi(\mu;\mu) = \frac{1}{r}\sum_{T \in \inputT} \Phi_S(T; \mu)$.  Thus, for each $s_i \in S$, the corresponding coordinate values on the left and right side of this equation must be equal.  The coordinate value on the left side is just the weight of $s_i$ in the mean, $|s_i|_{\mu}$.  This weight is strictly positive, so each coordinate value on the right side must be strictly positive.  Thus, for all $1 \leq i \leq m$, we have:

\begin{align*}
\sum_{T \in \inputT} \left ( \sum_{s_i \in E (T)} |s_i|_T + \sum_{s_i \in X(T)} -\frac{||B_{p(i,T)}^T||}{||A_{p(i,T)}^T||}|s_i|_{\mu} \right ) &> 0 \\
    \sum_{\substack{T \in \inputT:\\s_i \in E(T)}} |s_i|_T &> \sum_{\substack{T \in \inputT:\\ s_i \in X(T)}} \frac{||B_{p(i)}^T||}{||A_{p(i)}^T||}|s_i|_{\mu}
\end{align*}
\end{proof}

\subsection{A Sufficient Condition for a Split to be in the Mean}

We now give a sufficient condition for a split to be in the mean tree.  We will show that if the sum of the weights of split $s$ in all trees is greater than the sum of the weights of all other splits incompatible with it, then $s$ is in the mean.  

\begin{theorem}  \label{th:split_in_mean}
For a set of trees $\inputT$ in treespace $\Tn$, let $s$ be a split in at least one of the trees in $\inputT$. If
\begin{align*}
    \sum_{T \in \inputT} |s|_T > \sum_{x \in E(\inputT) \cap X(s)} \; \sum_{T \in \inputT} |x|_T
\end{align*}
then $s$ is in the mean of $\inputT$.
\end{theorem}

To prove this theorem, we will prove two contra-positive lemmas showing that if split $s$ is not in the mean, then it must satisfy the opposite inequality.  The first lemma covers the case when $s$ is incompatible with at least one split in the mean.  The second lemma covers the case when $s$ is compatible with all splits in the mean.  Each case corresponds to one of the two parts of Theorem~\ref{th:mean_characterization}.  Putting the lemmas together gives us the above theorem.  Recall that $X(s)$ is the set of splits incompatible with a split $s$.

%
%

\begin{lemma}  \label{lem:split_sum_incompatible}
For a set of trees $\inputT$ in treespace $\Tn$, let $\mu$ be their mean tree.  Let $s \in E(\inputT)$ be a split that is incompatible with $E(\mu)$.  Then 
\begin{align*}
    \sum_{T \in \inputT} |s|_T < \sum_{x \in E(\inputT) \cap X(s)} \; \sum_{T \in \inputT} |x|_T
\end{align*}
\end{lemma}

\begin{proof}   
Assume that $\mu$ contains exactly the splits $S = \{s_1, s_2, ..., s_m\}$ with positive edge weights. 
Since $s$ is incompatible with at least one of these splits, without loss of generality, assume that $s$ is incompatible with the splits $\{s_1, ..., s_{\ell}\}$, where $\ell \leq m$.



Following the notation of Lemma~\ref{lem:coord_sums_are_positive}, let $({\cal A}^T,{\cal B}^T)$ be the support of the geodesic from $\mu$ to tree $T \in \inputT$, where ${\cal A}^T = (A_0^T, A_1^T,..., A_{k(T)}^T)$ and ${\cal B}^T = (B_0^T, B_1^T, ..., B_{k(T)}^T)$. For each $1 \leq i \leq \ell$ and $T \in \inputT$, let $p(i,T)$ be the subscript of the partition of ${\cal A}^T$ containing $s_i$.   Then, by Lemma~\ref{lem:coord_sums_are_positive}, for each $1 \leq i \leq \ell$, 

\begin{align*}
\sum_{\substack{T \in \inputT:\\s_i \in E(T)}} |s_i|_T &> \sum_{\substack{T \in \inputT:\\ s_i \in X(T)}} \frac{||B_{p(i,T)}^T||}{||A_{p(i,T)}^T||}|s_i|_{\mu}.
\end{align*}


Adding up the above inequalities for all splits incompatible with $s$:

\begin{align} \label{eq:sum_incomp_splits}
	 \sum_{i = 1}^{\ell} \sum_{\substack{T \in \inputT:\\s_i \in E(T)}} |s_i|_T &> \sum_{i=1}^{\ell} \sum_{\substack{T \in \inputT:\\ s_i \in X(T)}} \frac{||B_{p(i,T)}^T||}{||A_{p(i,T)}^T||} |s_i|_{\mu}. 
\end{align}

Since the splits $\{s_1, ..., s_{\ell}\}$ are all incompatible with $s$ by definition, all weights on the left side of the inequality are for splits incompatible with $s$.  Thus, we can add the weights of all other splits in $E(\inputT)$ incompatible with $s$ to the left side of the inequality:
\begin{align*}
	\sum_{x \in E(\inputT) \cap X(s)} \; \sum_{T \in \inputT} |x|_T \geq \sum_{i = 1}^{\ell} \sum_{\substack{T \in \inputT:\\s_i \in E(T)}} |s_i|_T > \sum_{i=1}^{\ell} \sum_{\substack{T \in \inputT:\\ s_i \in X(T)}} \frac{||B_{p(i,T)}^T||}{||A_{p(i,T)}^T||} |s_i|_{\mu}.
\end{align*}

Also since $\{s_1, ..., s_{\ell}\}$ are all incompatible with $s$, any tree $T \in \inputT$ that contains the split $s$ is included in the right hand sum exactly $\ell$ times.  We can thus reduce the right hand side of the inequality by restricting the second summation to be over only those trees containing $s$:

\begin{align*}
	\sum_{x \in E(\inputT) \cap X(s)} \; \sum_{T \in \inputT} |x|_T > \sum_{i=1}^{\ell} \sum_{\substack{T \in \inputT:\\ s_i \in X(T)}} \frac{||B_{p(i,T)}^T||}{||A_{p(i,T)}^T||} |s_i|_{\mu}  > \sum_{i=1}^{\ell} \sum_{\substack{T \in \inputT:\\ s \in E(T)}} \frac{||B_{p(i,T)}^T||}{||A_{p(i,T)}^T||} |s_i|_{\mu}. 
\end{align*}
Note that at least one tree in $\inputT$ does not contain split $s$ or else split $s$ would be in the mean tree by Proposition~\ref{prop:edges_common_to_all_trees} yielding the strict inequality.

Next, we switch the sums on the right hand side of the inequality:  
\begin{align}
\label{eq:half_way_through}
	\sum_{x \in E(\inputT) \cap X(s)} \; \sum_{T \in \inputT} |x|_T >  \sum_{\substack{T \in \inputT:\\ s \in E(T)}} \sum_{i=1}^{\ell} \frac{||B_{p(i,T)}^T||}{||A_{p(i,T)}^T||} |s_i|_{\mu}.
\end{align}

For each tree $T \in \inputT$ containing split $s$, let $q(T)$ be the index of the partition of ${\cal B}^T$ containing $s$.  Note that $q(T) > 0$ since $s$ is incompatible with the mean $\mu$.  By \cite[Theorem 3.6]{Owen2011}, 
for any support pair $(A_j, B_j)$, each split in $B_j$ is incompatible with at least one split in $A_j$ and vice versa.  Therefore, for each tree $T \in \inputT$, $A_{q(T)}$ contains at least one $s_i$, $1 \leq i \leq \ell$, and so 

\begin{align}
\label{eq:only_A_q(T)}
    \sum_{\substack{T \in \inputT:\\ s \in E(T)}} \sum_{i=1}^{\ell} \frac{||B_{p(i,T)}^T||}{||A_{p(i,T)}^T||} |s_i|_{\mu} \geq  \sum_{\substack{T \in \inputT:\\ s \in E(T)}} \sum_{\substack{i \in \{1, ..., \ell\}:\\ s_i \in A_{q(T)}}} \frac{||B_{q(T)}^T||}{||A_{q(T)}^T||} |s_i|_{\mu}.
\end{align}

We now want to show that $\sum_{\substack{i \in \{1, ..., \ell\}:\\ s_i \in A_{q(T)}}} \frac{||B_{q(T)}^T||}{||A_{q(T)}^T||} |s_i|_{\mu} \geq |s|_T$.  For each tree $T \in \inputT$ containing split $s$, $s$ is incompatible with either all or only some splits in $A_{q(T)}^T$.  

First consider the case where $s$ is incompatible with all splits in $A_{q(T)}^T$, which implies $A_{q(T)}^T \subseteq \{s_1, ..., s_{\ell}\}$.  
%
Then

\begin{align*}
    \sum_{\substack{i \in \{1, ..., \ell\}:\\ s_i \in A_{q(T)}}} \frac{||B_{q(T)}^T||}{||A_{q(T)}^T||} |s_i|_{\mu} = \frac{\left (\sum_{x \in A_{q(T)}^T} |x|_{\mu} \right)||B_{q(T)}^T||}{||A_{q(T)}^T||} \geq ||B_{q(T)}^T|| \geq |s|_{T}
\end{align*}

where the first inequality follows from $\sum_{x \in A_{q(T)}^T} |x|_{\mu} \geq \sqrt{\sum_{x \in A_{q(T)}^T} |x|_{\mu}^2} = ||A^T_{q(T)}||$.

Next consider the second case where $s$ is incompatible with only some splits in $A_{q(T)}^T$.  Let $C_1 =  A^T_{q(T)} \cap \{s_1, ..., s_{\ell}\}$ be the set of splits incompatible with $s$, and let $C_2 = A^T_{q(T)} \setminus C_1$ be the set of splits compatible with $s$.  Then
\begin{align*}
    \sum_{\substack{i \in \{1, ..., \ell\}:\\ s_i \in A_{q(T)}}} \frac{||B_{q(T)}^T||}{||A_{q(T)}^T||} |s_i|_{\mu} = \frac{\left (\sum_{x \in C_1} |x|_{\mu} \right)||B_{q(T)}^T||}{||A_{q(T)}^T||} \geq \frac{||C_1||||B_{q(T)}^T||}{||A_{q(T)}^T||}
\end{align*}
where the last inequality follows from  $\sum_{x \in C_1} |x|_{\mu} \geq \sqrt{ \sum_{x \in C_1} |x|_{\mu}^2} = ||C_1||$.  Now since $\left (A^T_{q(T)}, B^T_{q(T)} \right)$ is a support pair in a geodesic, Property P3 in Theorem~\ref{th:geo_characterization} must hold.  Let $D_1 = s$ and $D_2 = B_{q(T)}^T \backslash s$.  The splits $C_2 \cup D_1 = C_2 \cup s$ are mutually compatible by definition of $C_2$.  Thus, Property P3 implies that $\frac{||C_1||}{||D_1||} > \frac{||C_2||}{||D_2||}$ or $\frac{||C_1||}{|s|_T} > \frac{||C_2||}{|| B^T_{q(T)} \backslash s||}$.  Squaring both sides yields:
\begin{align}
\label{eq:P3}
    \frac{||C_1||^2}{|s|_T^2} > \frac{||C_2||^2}{|| B^T_{q(T)} \backslash s||^2}.
\end{align}
We note that 
\begin{align}
\label{eq:sub1}
||A^T_{q(T)}||^2 = \sum_{x \in A^T_{q(T)}} |x|_{\mu}^2 = \sum_{x \in C_1} |x|_{\mu}^2 + \sum_{x \in C_2} |x|_{\mu}^2 = ||C_1||^2 + ||C_2||^2
\end{align}
and that
\begin{align}
\label{eq:sub2}
|| B^T_{q(T)} \backslash s||^2 = \sum_{x \in B^T_{q(T)} \backslash s} |x|_{T}^2 = \left ( \sum_{x \in B^T_{q(T)} } |x|_{T}^2 \right ) - |s|_T^2 = || B^T_{q(T)}||^2 - |s|_T^2
\end{align} 
Cross-multiplying Equation~\ref{eq:P3} and making the substitutions given by Equations~\ref{eq:sub1} and \ref{eq:sub2}, we get:
\begin{align*}
\label{eq:case2}
    \nonumber ||C_1||^2 \left(|| B^T_{q(T)}||^2 - |s|_T^2 \right) &> |s|_T^2 \left( ||A^T_{q(T)}||^2 - ||C_1||^2 \right) \\
    \nonumber ||C_1||^2|| B^T_{q(T)}||^2 - ||C_1||^2 |s|_T^2 &> |s|_T^2||A^T_{q(T)}||^2 - |s|_T^2||C_1||^2 \\
    \nonumber \frac{||C_1||^2|| B^T_{q(T)}||^2}{||A^T_{q(T)}||^2} &> |s|_T^2 \\
    \frac{||C_1|| || B^T_{q(T)}||}{||A^T_{q(T)}||} &> |s|_T,
\end{align*}
%
%
%
This concludes the second case.  We have shown for every $T \in \inputT$ containing split $s$, that

\begin{align*}
    \sum_{\substack{i \in \{1, ..., \ell\}:\\ s_i \in A_{q(T)}}} \frac{||B_{q(T)}^T||}{||A_{q(T)}^T||} |s_i|_{\mu}> |s|_{T}.
\end{align*}
Finally, we use this inequality to relax the right-hand side of Equation~\ref{eq:only_A_q(T)}, giving us: 
\begin{align*}
\sum_{x \in E(\inputT) \cap X(s)} \; \sum_{T \in \inputT} |x|_T >  \sum_{\substack{T \in \inputT:\\ s \in E(T)}} \sum_{i=1}^{\ell} \frac{||B_{p(i,T)}^T||}{||A_{p(i,T)}^T||} |s_i|_{\mu} \geq \sum_{\substack{T \in \inputT:\\ s \in E(T)}} \sum_{s_i \in A_{q(T)}} \frac{||B_{q(T)}^T||}{||A_{q(T)}^T||} |s_i|_{\mu} > \sum_{\substack{T \in \inputT:\\ s \in E(T)}} |s|_T.
\end{align*}
as desired.
\end{proof}   

%
%

The second lemma focuses on the case where the split $s$ is compatible with all splits in the mean, yet not in the mean itself.  

\begin{lemma}  \label{lem:split_sum_compatible}
For a set of trees $\inputT$ in treespace $\Tn$, let $\mu$ be their mean tree.  Let $s \in E(\inputT)$ be a split that is compatible with $E(\mu)$, but not in $E(\mu)$.  Then 
\begin{align*}
    \sum_{T \in \inputT} |s|_T \leq \sum_{x \in E(\inputT) \cap X(s)} \; \sum_{T \in \inputT} |x|_T.
\end{align*}
\end{lemma}

\begin{proof}
Assume that $\mu$ contains exactly the splits $S = \{s_1, s_2, ..., s_m\}$ with positive edge weight.  Since $s$ is compatible with all of these splits but not one of them, and a set of mutually compatible splits on $n$ leaves can have at most $n-3$ elements \cite{buneman1971}, then $m < n-3$.  Thus, the mean does not lie in a top dimensional orthant.  Let $w_s$ be the unit vector in the direction of $\jmath(s)$.  By part (i) of Theorem~\ref{th:mean_characterization} with $F = \{s\}$, we get:

\begin{align*}
\left \langle w_s, \sum_{T \in \inputT} \Psi_{S \cup s}(T,w_s; \mu) \right \rangle \leq 0.
\end{align*}

Note that for compactness, we will slightly abuse notation by writing $s$ instead of $\{s\}$ throughout this proof.

For each $T \in \inputT$, by Theorem~\ref{th:directional_limit}, part (ii), there exists some $\epsilon > 0$ such that the geodesic from $\mu(\lambda,w_s)$ to $T$ has the same geodesic support for all $0 < \lambda \leq \epsilon$.  Let this support be $$({\cal A},{\cal B}) = ((A_0^T, ..., A_{k(T)}^T),(B_0^T, ..., B_{k(T)}^T)).$$  Notice that $w_s$ is a unit vector in which the only non-zero coordinate values corresponds to splits $S \cup s$, and the coordinate value corresponding to split $s$ is strictly positive.  Thus, we can apply Equation~\ref{eq:projection_of_Psi_coords_with_w_E_cup_F} to get the coordinates of $\Psi_{S \cup s}(T,w_s; \mu)$: 
\begin{align*}
    \Psi_{S \cup s}(T,w_s;\mu) = \jmath \left ( P_{B_0^T \cap (S \cup s)}(T),-\frac{||P_{B_1^T}(T)||}{||W_1||}W_1, ..., -\frac{||P_{B_{k(T)}^T}(T)||}{||W_{k(T)}||}W_{k(T)} \right ).
\end{align*}
where $W_j = P_{A_i^T \cap S}(\mu)$, unless $P_{A_j^T \cap S}(\mu) = 0$, in which case $W_j = P_{A_j^T \cap s}(w_s)$, and $\jmath$ is the linear transformation defined in Definition~\ref{def:jmath}.

Consider the first coordinates $P_{B_0^T \cap (S \cup s)}(T)$.  The partition $B_0^T$ contains all splits common to trees $\mu$ and $T$, including those contained in only one tree that are compatible with all splits in the other tree.  Therefore, the only non-zero coordinates in $P_{B_0^T \cap (S \cup s)}(T)$ will be splits in $T$ that are also in $S \cup s$.  Thus, the non-zero coordinate values contributed by the term $P_{B_0^T \cap (S \cup s)}(T)$ are exactly $|e|_T$ for each $e \in (S \cup s) \cap E(T)$.

Next, for all $1 \leq j \leq k(T)$, consider the remaining coordinates $-\frac{||P_{B_j^T}(T)||}{||W_j||}W_j$.  By definition, $B_j^T$ only contains edges in $T$ with positive weight, and thus $||P_{B_j^T}(T)|| = ||B_j^T||$.  Also by definition, if $A_j^T \cap S \neq \emptyset$, then $W_j = P_{A_j^T \cap S}(\mu)$, and so $||W_j|| = || A_j^T \cap S||$.  If $A_j^T \cap S = \emptyset$ instead, then $W_j = P_{A_j^T \cap s}(w_s) = w_s$, the vector with 1 in the coordinate corresponding to split $s$.  Thus, 
\begin{equation}
\label{eq:cases_for_W_fraction_coords}
-\frac{||P_{B_j^T}(T)||}{||W_j||}W_j = \begin{cases}
-\frac{||B_j^T||}{||A^T_j\cap S||}P_{A_j^T \cap S}(\mu), & \text{if $A^T_j \cap S \neq \emptyset$} \\
-\frac{||B_j^T||}{||w_s||}w_s = -||B_j^T||w_s, & \text{if $A^T_j \cap S = \emptyset$}
\end{cases}
\end{equation}

%
%
%
We want to compute $\left \langle w_s, \sum_{T \in \inputT} \Psi_{S \cup s}(T,w_s; \mu) \right \rangle$, which is the dot product of $\sum_{T \in \inputT} \Psi_{S \cup s}(T,w_s; \mu)$ with $w_s$.  Since the only non-zero value in $w_s$ is a 1 in the coordinate corresponding to $s$, this dot product is the sum of the $s$ coordinate values in $\Psi_{S \cup s}(T,w_s; \mu)$ for all trees $T \in \inputT$.

For tree $T \in \inputT$, let $p(T)$ be the index of the partition of ${\cal A}^T$ containing $s$.  If $p(T) = 0$, then $T$ contains $s$ and the $s$ coordinate value in $\Psi_{S \cup s}(T,w_s; \mu)$ is $|s|_T$.   Otherwise, if $A_{p(T)}$ contains any $s_i$, for $1 \leq i \leq m$, then $A_{p(T)} \cap S \neq \emptyset$, implying $-\frac{||P_{B_{p(T)}^T}(T)||}{||W_{p(T)}||}W_{p(T)} = -\frac{||B_{p(t)}^T||}{||A^T_{p(t)}\cap S||}P_{A_{p(T)}^T \cap S}(\mu)$ by Equation~\ref{eq:cases_for_W_fraction_coords}.  Since $s \notin A_j^T \cap S$, this term will have a 0 in the $s$ coordinate and tree $T$ will contribute 0 to the sum $\sum_{T \in \inputT} \Psi_{S \cup s}(T,w_s; \mu)$.
However, if $A_{p(T)}$ does not contain any $s_i$, for $1 \leq i \leq m$, then $A_{p(T)} = \{s\}$, implying $A_{p(T)} \cap S = \emptyset$.  Then by Equation~\ref{eq:cases_for_W_fraction_coords}, $-\frac{||P_{B_{p(T)}^T}(T)||}{||W_{p(T)}||}W_{p(T)} = -||B_{p(T)}^T||w_s$, and tree $T$ contributes $-||B_{p(T)}^T||$ to the sum $\sum_{T \in \inputT} \Psi_{S \cup s}(T,w_s; \mu)$.
%
%
Therefore, the coordinate value for split $s$ in $\sum_{T \in \inputT} \Psi_{S \cup s}(T,w_s;\mu)$ is: 

\begin{align*}
\sum_{\substack{T \in \inputT:\\s \in  E(T)}} |s|_T 
- \sum_{\substack{T \in \inputT:\\s \notin E(T), \\A_{p(T)} \cap S \neq \emptyset}} 0
- \sum_{\substack{T \in \inputT:\\s \notin E(T), \\A_{p(T)} \cap S = \emptyset}} ||B_{p(T)}^T || 
= \sum_{\substack{T \in \inputT:\\s \in E(T)}} |s|_T - \sum_{\substack{T \in \inputT:\\s \notin E(T), \\A_{p(T)} = \{s\}}} ||B_{p(T)}^T ||
\end{align*}
and thus
\begin{align*}
    \left \langle w_s, \sum_{T \in \inputT} \Psi_{S \cup s}(T,w_s; \mu) \right \rangle =  \sum_{\substack{T \in \inputT:\\s \in E(T)}} |s|_T - \sum_{\substack{T \in \inputT:\\s \notin E(T), \\A_{p(T)} = \{s\}}} ||B_{p(T)}^T ||.
\end{align*}
Substituting this expression back into part (i) of Theorem~\ref{th:mean_characterization}, we get 

\begin{align*}
\sum_{\substack{T \in \inputT:\\s \in E(T)}} |s|_T - \sum_{\substack{T \in \inputT:\\s \notin E(T), \\A_{p(T)} = \{s\} }} ||B_{p(T)}^T|| &\leq 0 \\
\sum_{\substack{T \in \inputT:\\s \in E(T)}} |s|_T &\leq \sum_{\substack{T \in \inputT:\\s \notin E(T), \\A_{p(T)} = \{s\} }} ||B_{p(T)}^T|| \\
&\leq \sum_{\substack{T \in \inputT:\\s \notin E(T), \\ A_{p(T)} = \{s\} }} \sum_{x \in B_{p(T)}^T} |x|_T
\end{align*}
where the last inequality follows from $||B_{p(T)}^T|| = \sqrt{\sum_{x \in B_{p(T)}^T} |x|_T^2} \leq \sum_{x \in B_{p(T)}^T} |x|_T$.

If $A_{p(T)} = \{s\}$, then all splits in $B_{p(T)}^T$ are incompatible with $s$ by \cite[Theorem 3.6]{Owen2011}.  This observation implies we can relax the right-hand side of the above inequality as follows:

\begin{align*}
    \sum_{\substack{T \in \inputT:\\s \in E(T)}} |s|_T \leq \sum_{\substack{T \in \inputT:\\s \notin E(T), \\ A_{p(T)} = \{s\} }} \sum_{x \in B_{p(T)}^T} |x|_T \leq \sum_{\substack{T \in \inputT:\\s \notin E(T), \\ A_{p(T)} = \{s\} }} \sum_{x \in X(s) \cap E(T)} |x|_T \leq \sum_{x \in E(\inputT) \cap X(s)} \sum_{T \in \inputT} |x|_T.
\end{align*}
%
%
\end{proof}

Theorem~\ref{th:split_in_mean} follows directly from Lemmas~\ref{lem:split_sum_incompatible} and~\ref{lem:split_sum_compatible}.  It is the first theorem to give a condition for a split to be in the mean beyond it appearing in all input trees.  To use this theorem, we define a quantity for each split, called the \emph{split sum}:

\begin{definition}
Fix a split $s$.  Then the \emph{split sum} of $s$ for input trees $\inputT$, $\sigma(s,\inputT)$ is 
\begin{align*}
    \sigma(s,\inputT) = \sum_{T \in \inputT} |s|_T -  \sum_{x \in E(\inputT) \cap X(s)} \quad \sum_{T \in \inputT} |x|_T.
\end{align*}
\end{definition}

Rephrasing Theorem~\ref{th:split_in_mean} using the split sum $\sigma(s,\inputT)$, we have:
\begin{cor}
Let $s$ be a split in $\Tn$ and $\inputT$ a set of trees in $\Tn$.  Then $\sigma(s,\inputT) > 0$ implies $s$ is a split in the Fr\'echet mean of $\inputT$.
\end{cor}

Note that even if the split sum is negative for all splits in the set of input trees, this does not imply that the mean does not contain any of these splits and is at the origin.  We now give an example of such a scenario.

\begin{example} \label{ex:split_in_mean_with_neg_axis_sum}
We show there is a set of tree $\inputT$ of four trees in $\mathbb{T}_5$ such that there is no split $s$ with positive split sum $\sigma(s,\inputT)$, but the mean is not at the origin.  Consider the four input trees, $\inputT$, corresponding to the points in Figure~\ref{fig:ex_shivam_ella} and Example~\ref{ex:shivam_and_ella}.  Three of the trees have splits $s_1$ and $s_2$, with corresponding edge weights $(1,1)$, $(1,3)$, and  $(1,w)$.  The other tree has splits $s_3$ and $s_4$, with corresponding edge weights $(10,10)$.   The split sums are as follows:  
\begin{alignat*}
    \sigma(s_1,\inputT) &= 1+3 + w -10 - 10 &&= w-16 \\
    \sigma(s_2,\inputT) &= 1+1+1-10 &&= -7 \\
    \sigma(s_3,\inputT) &= 10-1-3-w &&= 6-w \\
    \sigma(s_4,\inputT) &= 10-1-1-1-1-3-w &&= 3-w.
\end{alignat*}

If $w = 10$ then all of these split sums are negative, and the split sums of all other splits non-positive.  However, from Example~\ref{ex:shivam_and_ella}, the mean is in the quadrant with axes $s_1$ and $s_2$.

\end{example}

Furthermore, it is even possible for a split to be the only split in the mean but not have a positive axis sum.  This counter-intuitive scenario is illustrated in the following example:

\begin{example}
In $\Tn$ with $n \geq 5$, consider a pair of input trees $\inputT = \{T_1, T_2\}$.  Suppose the tree $T_1$ has a single split $s_1$ with weight 6, and tree $T_2$ has exactly two splits $s_2$ and $s_3$, both of which are incompatible with $s_1$, and have weights 3 and 4, respectively.  Then the split sum of $s_1$ is $\sigma(s_1,\inputT) = 6 - 3 - 4 = -1$.  

We now compute the Fr\'echet mean.  By \cite[Corollary 4.1]{bhv}, because $T_1$ has no splits compatible with $T_2$, the geodesic between the two trees passes through the origin.  The Fr\'echet mean of $T_1$ and $T_2$ will be the mid-point of this geodesic.  The leg of the geodesic from $T_1$ to the origin lies along the axis corresponding to $s_1$ and has length 6, while the leg of the geodesic from the origin to $T_2$ has length $\sqrt{3^2 + 4^2} = 5$.  Therefore, the midpoint of the geodesic and Fr\'echet mean is the tree with single split $s_1$ with weight 0.5.
\end{example}



\subsection{A Necessary Condition for a Split to be in the Mean}

We now give a necessary condition for a split to be in the mean tree.  This condition is also an inequality on the split weights of the input trees, and states that the sum of the squares of the total weight of each mean split in the input trees must be greater than the the sum of the squared split weights of all splits incompatible with the mean.  As with the sufficient condition, we derive this inequality by assuming we know the mean, and then manipulating the expression to remove all references to the exact splits and weights of the mean.  

\begin{theorem}
\label{th:mean_implies}
Let $\mu$ be the mean tree of the input trees $\inputT = \{T_1, T_2, ..., T_r\}$ in treespace $\Tn$, and let $S = \{s_1, s_2, ..., s_m\}$ be the splits with positive weight in $\mu$. Then 
\begin{align}
\sum_{i = 1}^{m} \left ( \sum_{T \in \inputT} |s_i|_T \right)^2 > \sum_{x \in E(\inputT) \cap X(\mu) } \; \sum_{T \in \inputT} |x|_T^2. 
\end{align}
\end{theorem}

\begin{proof}
Let $\mu$ be the Fr\'echet mean of $\inputT$, and suppose it is in the interior of the orthant corresponding to splits $s_1, ..., s_m$.  For tree $T \in \inputT$, let $({\cal A}^T,{\cal B}^T)$, where ${\cal A}^T = (A_0^T, A_1^T,..., A_{k(T)}^T)$ and ${\cal B}^T = (B_0^T, B_1^T, ..., B_{k(T)}^T)$, be the support of the geodesic from the mean $\mu$ to tree $T$.  For each $1 \leq i \leq m$ and each $T \in \inputT$, let $p(i,T)$ be the subscript of the partition of ${\cal A}^T$ containing $s_i$.  Then by Lemma~\ref{lem:coord_sums_are_positive}, 
\begin{align}
\sum_{\substack{T \in \inputT:\\s_i \in E(T)}} |s_i|_T &> \sum_{\substack{T \in \inputT:\\ s_i \in X(T)}} \frac{||B_{p(i,T)}^T||}{||A_{p(i,T)}^T||}|s_i|_{\mu}
\end{align}
for each $1 \leq i \leq m$.  Using the convention that $|s_i|_T = 0$ if split $s_i$ is not in tree $T$, we can slightly simplify the above expression to

\begin{align*}
    \sum_{T \in \inputT} |s_i|_T &> \sum_{\substack{T \in \inputT:\\ s_i \in X(T)}} \frac{||B_{p(i,T)}^T||}{||A_{p(i,T)}^T||}|s_i|_{\mu}.
\end{align*}

Squaring each side we get:

\begin{align*}
    \left ( \sum_{T \in \inputT} |s_i|_T \right)^2 &> \left ( \sum_{\substack{T \in \inputT:\\ s_i \in X(T)}} \frac{||B_{p(i,T)}^T||}{||A_{p(i,T)}^T||}|s_i|_{\mu} \right)^2 \\
    &\geq \sum_{\substack{T \in \inputT:\\ s_i \in X(T)}} \frac{||B_{p(i,T)}^T||^2}{||A_{p(i,T)}^T||^2}|s_i|^2_{\mu}.
\end{align*}

Summing up the inequalities for all $1 \leq i \leq m$, we get:

\begin{align}
\label{eq:mean_implies_helper1}
    \sum_{i = 1}^{m} \left ( \sum_{T \in \inputT} |s_i|_T \right)^2 &> \sum_{i = 1}^{m}\sum_{\substack{T \in \inputT:\\s_i \in X(T)}} \frac{||B_{p(i,T)}^T||^2}{||A_{p(i,T)}^T||^2}|s_i|^2_{\mu}.
\end{align}

By Theorem~\ref{th:geo_characterization}, $A_0^T$ contains all splits that are either shared by both $T$ and $\mu$, or in one tree and compatible with the other.  Therefore, the splits $A_1^T \cup \cdots \cup A_{k(T)}^T$ are incompatible with tree $T$, and thus in $X(T)$.  This implies that for every $1 \leq j \leq k(T)$ and every split $a \in A_j^T$, the term $\frac{||B_j^T||^2}{||A_j^T||^2}|a|^2_{\mu}$ appears on the right hand side of the above inequality.  Thus, we can reorder the summations on the right hand side of this inequality:

\begin{align*}
    \sum_{i = 1}^{m}\sum_{\substack{T \in \inputT:\\s_i \in X(T)}} \frac{||B_{p(i,T)}^T||^2}{||A_{p(i,T)}^T||^2}|s_i|^2_{\mu} &= \sum_{T \in \inputT} \sum_{j = 1}^{k(T)} \sum_{a \in A_j^T} \frac{||B_j^T||^2}{||A_j^T||^2} |a|^2_\mu \\
    &= \sum_{T \in \inputT} \sum_{j = 1}^{k(T)} \frac{||B_j^T||^2}{||A_j^T||^2} \left( \sum_{a \in A_j^T} |a|^2_\mu \right)  \\
    &= \sum_{T \in \inputT} \sum_{j = 1}^{k(T)} ||B_j^T||^2
\end{align*}
where the final line follows from $||A_j^T||^2 = \sum_{a \in A_j^T} |a|_{\mu}^2$. 

We substitute this into Equation~\ref{eq:mean_implies_helper1} to get:

\begin{align}
\label{eq:mean_implies_helper2}
    \sum_{i = 1}^{m} \left ( \sum_{T \in \inputT} |s_i|_T \right)^2 > \sum_{T \in \inputT} \sum_{j = 1}^{k(T)} ||B_j^T||^2.
\end{align}
Noting that the set of splits $\cup_{j = 1}^{k(T)} B_j^T$ is exactly the set of splits in $T$ that are not also in the mean $\mu$ nor compatible with it, we get

\begin{align*}
	\sum_{T \in \inputT} \sum_{j = 1}^{k(T)} ||B_j^T||^2 &= \sum_{T \in \inputT} \; \sum_{x \in E(T) \cap X(\mu)} |x|_T^2 \\
	&= \sum_{x \in E(\inputT) \cap X(\mu) } \; \sum_{T \in \inputT} |x|_T^2.
\end{align*}
where in writing the last line, we make our usual assumption that $|x|_T = 0$ if $x \notin E(T)$.

Substituting this into Equation~\ref{eq:mean_implies_helper2}, we get
\begin{align*}
\sum_{i = 1}^{m} \left ( \sum_{T \in \inputT} |s_i|_T \right)^2 > \sum_{x \in E(\inputT) \cap X(\mu) } \; \sum_{T \in \inputT} |x|_T^2 
\end{align*}
\end{proof}

Unfortunately, we cannot relax either side to get either a sum of squares or a square of sums on both sides.  We can apply the Cauchy-Schwarz inquality to get the following corollary.  However, as it depends on the number of input trees, which will likely be large, it is of limited usefulness.

\begin{cor}
\label{cor:mean_implies_cauchy_schwarz}
Let $\mu$ be the mean tree of the input trees $\inputT = \{T_1, T_2, ..., T_r\}$ in treespace $\Tn$, and let $S = \{s_1, s_2, ..., s_m\}$ be the splits with positive weight in $\mu$. Then 
\begin{align}
r \sum_{i = 1}^{m} \sum_{T \in \inputT} |s_i|_T^2 > \sum_{x \in E(\inputT) \cap X(\mu) } \; \sum_{T \in \inputT} |x|_T^2 
\end{align}
and 
\begin{align}
r \sum_{i = 1}^{m} \left( \sum_{T \in \inputT} |s_i|_T \right)^2 > \sum_{x \in E(\inputT) \cap X(\mu) } \left( \sum_{T \in \inputT} |x| \right)^2
\end{align}
\end{cor}

\begin{proof}
By Theorem~\ref{th:mean_implies},
\begin{align*}
\sum_{x \in E(\inputT) \cap X(\mu) } \; \sum_{T \in \inputT} |x|_T^2 &< \sum_{i = 1}^{m} \left ( \sum_{T \in \inputT} |s_i|_T \right)^2  \\
&= \sum_{i = 1}^{m} \left ( \sum_{T \in \inputT} |s_i|_T \cdot 1 \right)^2 \\
& \leq \sum_{i = 1}^{m} \left( \sum_{T \in \inputT} |s_i|_T^2 \right) \left( \sum_{T \in \inputT} 1^2 \right) \\
&= r \sum_{i = 1}^{m} \sum_{T \in \inputT} |s_i|_T^2
\end{align*}
where the second last line follows by the Cauchy-Schwarz inequality.

Alternatively,
\begin{align*}
\sum_{i = 1}^{m} \left ( \sum_{T \in \inputT} |s_i|_T \right)^2 & > \sum_{x \in E(\inputT) \cap X(\mu) } \; \sum_{T \in \inputT} |x|_T^2  \\
&= \frac{1}{r}\sum_{x \in E(\inputT) \cap X(\mu) } \; \left( \sum_{T \in \inputT} |x|_T^2 \right) \left( \sum_{T \in \inputT} 1^2 \right) \\
& \geq \frac{1}{r}\sum_{x \in E(\inputT) \cap X(\mu) } \; \left( \sum_{T \in \inputT} |x|_T \cdot 1 \right)^2.
\end{align*}
\end{proof}

We now give a counter-example to show that even if the condition of Theorem~\ref{th:mean_implies} holds for splits $\{s_1, s_2, ..., s_m\}$, the mean may not be in the corresponding orthant.

\begin{example}
Consider four splits $s_1$, $s_2$, $s_3$, and $s_4$, where $s_1$ is compatible with $s_2$, $s_2$ is compatible with $s_3$, $s_3$ is compatible with $s_4$, and no other pairs of splits are compatible.  This is the same arrangements of splits as in Figure~\ref{fig:ex_shivam_ella}.  Let $T_1$ be the tree with exactly one split $s_1$ with weight 3, let $T_2$ be the tree with exactly one split $s_2$ with weight 3, and let $T_3$ be the tree with exactly splits $s_3$ and $s_4$ with weights 4 and 1, respectively.  The Theorem~\ref{th:mean_implies} condition holds for split set $\{s_1,s_2\}$ and $\{s_2,s_3\}$, since $|s_1|_{T_1}^2 + |s_2|_{T_2}^2  = 3^2 + 3^2 = 18 > 17 = 4^2 + 1^1 = |s_3|_{T_3}^2 + |s_4|_{T_3}^2$ and $|s_2|_{T_2}^2 + |s_3|_{T_3}^2  = 3^2 + 4^2 = 25 > 10 = 3^2 + 1^1 = |s_1|_{T_1}^2 + |s_4|_{T_3}^2$, respectively.  We lay the $(s_1, s_2)$, $(s_2,s_3)$, and $(s_3,s_4)$ orthants in the plane, as in Figure~\ref{fig:ex_shivam_ella}, and compute the Euclidean mean to get the point in the $(s_2,s_3)$ orthant with $s_2$ having weight $\frac{2}{3}$ and $s_3$ having weight $-\frac{1}{3}$.  Since that point is in a valid orthant, it is the mean tree.  Therefore, just because the splits $s_1$ and $s_2$ satisfy the necessary condition, did not mean the mean was in that orthant.
\end{example}



As with the sufficient condition, we define a quantity based on the necessary condition:

\begin{definition}
For a set of trees $\inputT$ in treespace $\Tn$, fix a set of splits $S \subset E(\inputT)$.  Then the \emph{square-sum difference}, $\ssd(S,\inputT)$ is
\begin{align}
    \ssd(S,\inputT) = \sum_{s \in S} \left ( \sum_{T \in \inputT} |s|_T \right)^2 - \sum_{x \in E(\inputT) \cap X(S) } \; \sum_{T \in \inputT} |x|_T^2 .
\end{align}
\end{definition}

Using this notation, Theorem~\ref{th:mean_implies} can be rewritten as follows.

\begin{cor}
\label{cor:ssd}
Let $\inputT = \{T_1, T_2, ..., T_r\}$ be a set of trees in treespace $\Tn$, and let $S$ be the splits with positive weight in the mean of $\inputT$. Then $\ssd(S,\inputT) > 0$.
\end{cor}

This corollary implies that if $\ssd(S,\inputT) \leq 0$ for some set $S$ of mutually compatible splits in $E(\inputT)$, then the mean is not in the interior of orthant $\Or(S)$.  However, the splits $S$ can still be in the mean if $\rho(S \cup F,\inputT) > 0$ for some set of splits $F$ such that splits $S \cup F$ are mutually compatible.  We can only say that a set of splits $S$ do not appear in the mean together if we augment $S$ with all splits compatible with $S$ to get $S' = S \cup ( C(S) \cap E(\inputT))$.  If $\rho(S',T) \leq 0$, then no subset of $S'$, in particular $S$, can be together in the mean, since any subset would have a smaller positive sum and larger negative sum in Equation~\ref{cor:ssd}, which is still negative.

\section{Applications for Computing the Mean}

The work from Miller {\em et al.}~\cite{miller2015} showed that if the mean lies in the interior of a top-dimensional orthant and the correct orthant is identified, the mean can be computed quickly.   Skwerer {\em et al.}~\cite{ skwererMeans} extended this work to include trees that lie on the boundaries of a top-dimensional orthant.  Therefore, the problem of computing the mean tree reduces to finding the correct orthant containing the mean tree. Our results give conditions for when splits must be and when splits are forbidden from being part of the mean tree, giving a pre-processing step to limit the number of orthants that need to be checked to find the location of the mean tree.

Our approach is to first find any common splits among the inputted trees, and from Lemma~\ref{lem:common_split}, decomposes the mean along them.  Since this can be done in polynomial time, we assume this is done before the remaining steps.  At the end, we reassemble the means of the subtrees, as explained in Lemma~\ref{lem:common_split}.  While we can iteratively look at each split of the inputted trees, and use these lemmas to affirm or deny their membership in the mean tree, it does not classify all possible splits.  So, while likely to reduce the number of orthants that must be explored, there still could be an exponential number.

\section{Conclusions and Future Work}

We have found the first results that give non-trivial conditions for including or excluding splits from the mean. Our conditions are combinatoral, where  previous work has been based in optimization.
The two conditions classify some but not all splits in terms of the mean.  While it is not obvious how to classify the splits that are not captured by our conditions, it looks possible to extend the sufficient condition for splits to be in the mean, since unlike the necessary conditions, we do not fully use all parts of the definition to classify splits.   This suggests promising work for extending these conditions to classify more of the splits and develop techniques that decompose along the splits that must be in the mean.
While this paper focuses on BHV treespace, the work extends to the more general case of orthant spaces \cite{miller2015} (where we replace ``trees'' by points, and axis compatibilities are given by a flag simplicial complex).


\section{Acknowledgments}

We would like to thank Dennis Barden, Louis Billera, Huiling Le, Sean Skwerer, and Ed Swartz for helpful discussions.  We would like to thank the American Museum of Natural History and the CUNY Advanced Science Research Center for hosting us for several meetings.  This work was funded by a Research Experience for Undergraduates (REU) grant from 
the US National Science Foundation (\#1461094 to St.~John and Owen) as well as collaboration grants from the Simons Foundation (to St.~John and to Owen).





\bibliographystyle{plain}
\bibliography{means.bib}







\end{document}